%% file: main.tex
\documentclass[11pt, a4paper]{article} 

\usepackage{lmodern}
\usepackage{microtype}

\usepackage[british]{babel}
\usepackage[utf8]{inputenc}

\usepackage{amssymb}
\usepackage{amsmath}
\usepackage{amsthm}
\usepackage{amsfonts}
\usepackage{mathtools}
\usepackage{enumerate}
\usepackage{xcolor}
\usepackage{graphicx}
\usepackage{subcaption}
\usepackage{bm}
\usepackage{dsfont}
\usepackage{booktabs}

\usepackage{algorithm}
\usepackage{algpseudocode}

\newtheorem{theorem}{Theorem}
\newtheorem{corollary}[theorem]{Corollary}
\newtheorem{lemma}[theorem]{Lemma}
\newtheorem{proposition}[theorem]{Proposition}

\theoremstyle{definition}
\newtheorem{definition}[theorem]{Definition}
\newtheorem{example}{Example}
\newtheorem*{remark*}{Remark}
  
\theoremstyle{plain}
\newtheorem{assumption}{Assumption}

\numberwithin{equation}{section}
\numberwithin{theorem}{section}
\numberwithin{algorithm}{section}

\usepackage[hidelinks]{hyperref}

\newcommand\bbR{\mathbb{R}}

\newcommand\bbC{\mathbb{C}}

\newcommand\raq{\mathcal{R}}
\newcommand\diag{\mathrm{diag}}

\renewcommand{\vec}[1]{\bm{#1}}

\newcommand{\norm}[1]{\left\lVert#1\right\rVert}
\newcommand{\abs}[1]{\left\lvert#1\right\rvert}

\DeclareMathOperator{\spread}{spread}

\DeclareMathOperator{\real}{Re}
\DeclareMathOperator{\imag}{Im}

\newcommand{\removespace}{}

\newenvironment{keywords}[1]{\begin{quotation}\footnotesize\textbf{Key words.}\ #1}{\end{quotation}}
\newenvironment{MSCcodes}[1]{\begin{quotation}\footnotesize\textbf{MSC codes.}\ #1}{\end{quotation}}

\graphicspath{{.}{./figures/}{./figures/angle_convergence/}{./figures/simplex_test/}{./figures/vec_plots/}}

\makeatletter
\let\@fnsymbol\@arabic%
\makeatother

\title{A complex-projected Rayleigh quotient iteration for targeting interior eigenvalues}
\author{Nils Friess\footnotemark[1]
	\and%
	Alexander D.~Gilbert\footnotemark[2]
  \and%
  Robert Scheichl\footnotemark[1]
}

\date{}

\begin{document}
\maketitle%

\footnotetext[1]{Institute for Mathematics \& Interdisciplinary Centre for Scientific Computing, Universität Heidelberg, 69120 Heidelberg, Germany.
\texttt{nils.friess@iwr.uni-heidelberg.de}, \texttt{r.scheichl@uni-heidelberg.de}}
\footnotetext[2]{School of Mathematics and Statistics, University of New South Wales, Sydney NSW 2052, Australia.\\
\texttt{alexander.gilbert@unsw.edu.au}}

\input{prqi_body}

\textbf{Acknowledgments.} Many thanks go to Chris Hart for the excellent foundations he laid for this paper more than 12 years ago and to Marco Marletta for the many useful suggestions during that period. This work is supported by Deutsche Forschungsgemeinschaft (German Research
Foundation) under Germany's Excellence Strategy EXC 2181/1 - 390900948 (the Heidelberg STRUCTURES Excellence Cluster).

\small
\bibliographystyle{abbrv}
\bibliography{bibliography}

\end{document}

%% file: prqi_body.tex
\begin{abstract}
We introduce a new \emph{Projected Rayleigh Quotient Iteration} aimed at improving the convergence behaviour of classic Rayleigh Quotient iteration (RQI) by incorporating approximate information about the target eigenvector at each step. While classic RQI exhibits local cubic convergence for Hermitian matrices, its global behaviour can be unpredictable, whereby it may converge to an eigenvalue far away from the target, even when started with accurate initial conditions. This problem is exacerbated when the eigenvalues are closely spaced. The key idea of the new algorithm is at each step to add a complex-valued projection to the original matrix (that depends on the current eigenvector approximation), such that the unwanted eigenvalues are lifted into the complex plane while the target stays close to the real line, thereby increasing the spacing between the target eigenvalue and the rest of the spectrum. Making better use of the eigenvector approximation leads to more robust convergence behaviour and the new method converges reliably to the correct target eigenpair for a significantly wider range of initial vectors than does classic RQI. We prove that the method converges locally cubically and we present several numerical examples demonstrating the improved global convergence behaviour. In particular, we apply it to compute eigenvalues in a band-gap spectrum of a Sturm-Liouville operator used to model photonic crystal fibres, where the target and unwanted eigenvalues are closely spaced. The examples show that the new method converges to the desired eigenpair even when the eigenvalue spacing is very small, often succeeding when classic RQI fails.
\end{abstract}



\section{Introduction}
Rayleigh Quotient Iteration (RQI) is a simple yet effective algorithm to compute a single eigenvalue and eigenvector pair (eigenpair) of a matrix. For Hermitian matrices, RQI converges locally cubically and converges globally for almost all starting vectors~\cite{Parlett1974}, but it suffers from a significant flaw, whereby the global convergence can be unpredictable. It is well-known that RQI may converge to the wrong eigenpair, i.e., one that is different to the target eigenpair, even when executed with a good approximation of the target eigenvector as the initial vector and an initial shift close to the target eigenvalue.
This problem is exacerbated in cases where the distance between the wanted eigenvalue and its neighbours is small (cf.\ the numerical examples in Section~\ref{sec:numerics}).

In this paper, we propose a new algorithm that modifies RQI to overcome this problem by exploiting any prior information known about the target eigenvector. The idea of the new method is as follows. Starting with an approximation of the target eigenvector, we perturb the original problem by adding a purely complex 
projection matrix constructed using this approximation, 
such that the unwanted eigenvalues are lifted into the complex plane, while the target eigenvalue remains close to the real line (see Figure~\ref{fig:perturbed_spectrum}), thereby artificially increasing the gap between the target eigenvalue and its neighbours. We then use the perturbed matrix in RQI, where in each step of the iteration we update the perturbation matrix using the previous eigenvector approximation
and gradually reduce the size of the perturbation.

\begin{figure}[t]
  \centering
  \includegraphics[width=0.8\textwidth]{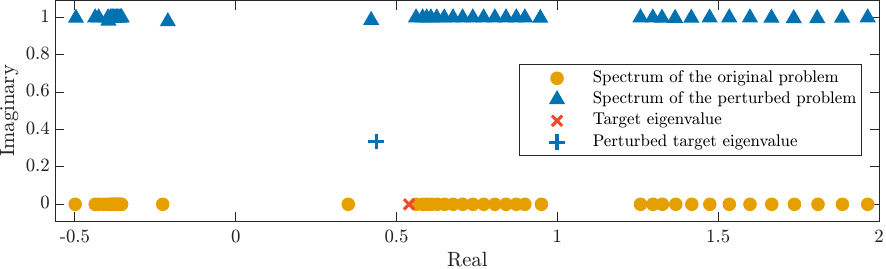}
  \caption{{Section of the spectrum of a matrix $A$ (orange circles) and the perturbed matrix $\widetilde{A}$ (blue triangles). Most of the perturbed eigenvalues $\widetilde{\lambda}$ have $\text{Im}({\widetilde{\lambda}}) \approx 1$ while one eigenvalue lies close to the real line. The matrix $A$ and the vector $\vec{u}$ that were used to construct $\widetilde{A}$ for this plot come from Example~\ref{ex:sturm} in Section~\ref{sec:numerics}.}}%
  \label{fig:perturbed_spectrum}
\end{figure}

More precisely, suppose we wish to compute an eigenpair $(\lambda, \vec{v}) \in \bbR \times \bbC^n$, $\vec{v} \neq \vec{0}$, of a Hermitian matrix $A \in \bbC^{n \times n}$ and we are given an approximation $\vec{u} \approx \vec{v}$, then we perturb the matrix $A$ by
\begin{equation}
  \label{eq:orig_form}
  A \longrightarrow A + i \gamma(I - \vec{u}\vec{u}^\ast)\,,
\end{equation}
where $i$ denotes the imaginary unit, the asterisk denotes the complex conjugate transpose and $I - \vec{u}\vec{u}^\ast$
is the projection onto the orthogonal complement of $\vec{u}$. The real scalar $\gamma > 0$ is used to control the magnitude of the perturbation. 

If $\vec{u}$ is a sufficiently good approximation of $\vec{v}$, then the perturbed target eigenvalue~$\widetilde{\lambda}$ stays near the real line while the remaining eigenvalues, $\mu \neq \lambda$, are perturbed by
\begin{equation*}
  \mu \longrightarrow \mu + i \gamma\,,
\end{equation*}
approximately. Since $\mu \in \bbR$, because $A$ is Hermitian, the perturbed eigenvalues now lie (approximately) along the line $i \gamma$ in the upper half of the complex plane, again see Figure~\ref{fig:perturbed_spectrum}, and so the distance between the perturbed target eigenvalue $\widetilde{\lambda}$ and the remainder of the perturbed spectrum is approximately $\gamma$.
Applying RQI to this problem, we will not compute the target eigenvalue but rather an eigenvalue of the perturbed matrix. Therefore, we replace $\gamma$ by a non-negative decreasing zero sequence ${(\gamma^{(k)})}_{k \ge 0} \subset \bbR$ such that the perturbation becomes increasingly small as $k \rightarrow \infty$ and we recover the original unperturbed problem. 

{In practice, the size of the perturbation $\gamma^{(k)}$ needs to be balanced. If it decays too
fast then we lose the effect of the complex perturbation and our method may not converge.
On the other hand, if it decays too slowly then it will slow down the convergence
to the target eigenpair.
A good balance is to choose $\gamma^{(k)}$ at each step 
to be the norm (or norm squared) of the current residual.
In this way, our algorithm utilises approximate eigenvector information to choose both 
the direction as well as the magnitude of the complex perturbation.
We prove that this choice leads to an asymptotic rate of convergence that is either quadratic or cubic, depending on the choice of the perturbations $\gamma^{(k)}$.}

The motivation comes from practical applications
where the eigenvalues are closely spaced but
some information about the target eigenvector is known.
For example, in models for photonic crystal fibres \cite{Kuchment2001} the spectrum has a band-gap structure and the eigenvalues that lie inside the gaps are very closely 
spaced; the corresponding eigenvectors are oscillatory and localised. 
As we show later in Section \ref{sec:numerics}, this limited, qualitative information about the target eigenvector
suffices to ensure that our method converges to it.
The method we propose was inspired by the approach in \cite{Marletta2012}, which works directly on the differential equation level. See also the recent article \cite{ovall2023}, which expands on that paper. Alternative numerical methods for computing localised eigenfunctions of differential operators are presented in~\cite{altmann2019,arnold2019,altmann2020}. In contrast to all those methods, our approach is purely algebraic, can be applied to any Hermitian matrix and requires only a one-line change in a classic RQI algorithm.

There are other approaches that modify RQI in order to make the convergence behaviour of the algorithm more predictable, see~\cite{pantazisszyld,szyld,beattiefox}. However, all those methods usually use \emph{a priori} information about the eigenvalues, e.g., the number of eigenvalues within a certain interval. Our approach is novel in that we exploit information about the eigenvectors instead.

The remainder of the paper is organised as follows. In the next section, we review convergence results for Inverse Iteration and classic RQI, since they form the basis for the convergence analysis of our method. In Section 3, we then introduce the new Projected Rayleigh Quotient Iteration (PRQI) method. We also prove there that the approach introduced above {in \eqref{eq:orig_form}} is equivalent to perturbing the matrix according to
\begin{equation*}
  A \longrightarrow A + i \gamma I\,.
\end{equation*}
That is, it is not necessary to perturb $A$ by the full projection matrix $i \gamma(I - \vec{u}\vec{u}^\ast)$ {as in~\eqref{eq:orig_form}} but it suffices to shift only the diagonal entries by $i \gamma$. This simplifies the implementation and also allows us to prove asymptotic convergence of our method in Section 3.2. Finally, in Section~\ref{sec:numerics} we present three numerical examples that demonstrate the efficacy of our method,
including a practical example from an 
application in photonics. 
In particular, we present results where our algorithm succeeds in locating the target eigenvalue within a closely-spaced spectrum.
Often a crude approximation is sufficient
for our method to succeed in converging to the target eigenpair,
while for the same initial 
conditions classic RQI fails.

\section{Classic Rayleigh Quotient Iteration}
We consider the Hermitian eigenvalue problem: find $(\lambda, \vec{v}) \in \bbR \times \bbC^n$, $\vec{v} \neq \vec{0}$, such that
\begin{equation}
  \label{eq:eigvalproblem}
  A\vec{v} = \lambda \vec{v}\,,
\end{equation}
where $A = A^\ast \in \bbC^{n \times n}$ is a complex Hermitian matrix. We denote the eigenvalues of $A$ by $\lambda_1, \dotsc, \lambda_n$ with corresponding eigenvectors $\vec{v}_1, \dotsc, \vec{v}_n$ which we assume to be normalised w.r.t.\ the Euclidean norm $\norm{\cdot} = \norm{\cdot}_2$. For simplicity, we assume the target eigenvalue is simple, i.e., it has algebraic multiplicity 1. The eigenvalues will usually be labelled such that $\lambda_1$ is the target eigenvalue (see, e.g., Theorem~\ref{thm:ii:convergence} below).

There are a wide variety of methods that can be used to solve~\eqref{eq:eigvalproblem} numerically, see, e.g., the monographs \cite{saad2011} and \cite{boermmehl}. One of the most widely used and simplest methods is the \emph{(shifted) inverse iteration}. Starting with a non-zero vector $\vec{x}^{(0)}$ and a shift $\mu \in \bbR$, one computes
\begin{equation}
  \label{eq:sii}
  \vec{x}^{(k+1)} = \alpha^{(k)} {(A - \mu I)}^{-1} \vec{x}^{(k)}\,,\quad k=0,1,\dotsc
\end{equation}
where $I$ denotes the $n \times n$ identity matrix and $\alpha^{(k)}$ is a scalar responsible for normalising $\vec{x}^{(k+1)}$. The sequence ${(\vec{x}^{(k)})}_k$ converges linearly to the eigenvector that corresponds to the eigenvalue closest to the shift $\mu$. More precisely, we have the following convergence result. Here, $\theta(\vec{u},\vec{w})$ denotes the angle between the vectors $\vec{u}$ and $\vec{w}$, i.e.,
\begin{equation}
\label{eq:angle_def}
  \theta(\vec{u},\vec{w}) = \arccos\bigg(\frac{\abs{\vec{u}^\ast \vec{w}}}{\norm{\vec{u}}\norm{\vec{w}}}\bigg)\,.
\end{equation}

\begin{theorem}[\cite{boermmehl}, Theorem 4.10]%
  \label{thm:ii:convergence}
  Let $\mu \in \bbR$ be such that $A - \mu I$ is invertible, i.e., $\mu$ is not an eigenvalue of $A$, and label the eigenvalues of $A$ such that 
  \begin{equation*}
    \abs{\lambda_1 - \mu} < \abs{\lambda_2 - \mu} \le \cdots \le \abs{\lambda_n - \mu}\,.
  \end{equation*}
  Then, for the sequence ${(\vec{x}^{(k)})}_{k \ge 0}$ generated by Inverse Iteration \eqref{eq:sii}, we have
  \begin{equation*}
    \tan\big(\theta(\vec{x}^{(k+1)}, \vec{v}_1) \big) \le \frac{\abs{\lambda_1 - \mu}}{\abs{\lambda_2 - \mu}}\tan\big(\theta(\vec{x}^{(k)}, \vec{v}_1)\big)\,.
  \end{equation*}
\end{theorem}

Inverse Iteration as given in~\eqref{eq:sii} only yields approximations of an eigenvector. To obtain an approximation of the corresponding eigenvalue the Rayleigh Quotient can be used. 
\begin{definition}[Rayleigh Quotient]
  Let $A \in \bbC^{n \times n}$. The mapping
  \begin{equation*}
    \raq_{A} : \bbC^n \setminus \{\vec{0}\} \rightarrow \bbC, \qquad
    \vec{x} \mapsto \frac{\vec{x}^\ast A \vec{x}}{\vec{x}^\ast \vec{x}}
  \end{equation*}
  is called the \emph{Rayleigh Quotient} corresponding to the matrix $A$.
\end{definition}
The following result specifies the quality of the Rayleigh Quotient as an estimate
for the eigenvalue.
\begin{theorem}[\cite{boermmehl}, Theorem 4.6]%
  \label{thm:quadacc}
  Let $A \in \bbC^{n \times n}$ be Hermitian, let   $\vec{x} \in \bbC^n$ be a nonzero vector and let $(\lambda, \vec{v})$ be an eigenpair of $A$. Then
  \begin{equation}
  \label{eq:RQ_quad}
    \abs{\raq_{A}(\vec{x}) - \lambda} \le \norm{A - \lambda I} \tan^2\big( \theta(\vec{x}, \vec{v})\big)\,,
  \end{equation}
  that is, the Rayleigh Quotient is a quadratically accurate estimate of an eigenvalue.
\end{theorem}

It is now quite natural to replace the fixed shift $\mu$ in inverse iteration by the Rayleigh Quotient of the current vector iterate. The resulting algorithm is called \emph{Rayleigh Quotient Iteration} (RQI) and is summarised in Algorithm~\ref{alg:rqi}.
\begin{algorithm}[t]
  \caption{Classic Rayleigh Quotient Iteration}\label{alg:rqi}
  \begin{algorithmic}[1]
    \Require $\vec{x} \in \bbC^n$ with $\|\vec{x}^{(0)}\| = 1$

    \For{$k=0,1,2,\dotsc$, until convergence}
      \State $\mu^{(k)} \gets {(\vec{x}^{(k)})}^\ast A \vec{x}^{(k)}$
      \Comment{Compute the Rayleigh Quotient}
      \State Solve $(A - \mu^{(k)}I)\vec{y}^{(k+1)} = \vec{x}^{(k)}$ for $\vec{y}^{(k+1)}$
      \State $\vec{x}^{(k+1)} \gets \vec{y}^{(k+1)} / \|{\vec{y}^{(k+1)}}\|$
      \Comment{Normalise}
    \EndFor
\end{algorithmic}
\end{algorithm}
Note that in the computation of the Rayleigh Quotient in Step 2, dividing by ${(\vec{x}^{(k)})}^\ast \vec{x}^{(k)}$ has been omitted, because the vector iterates are always normalised. 

In practice, the iteration is stopped when the 
the approximation is deemed to have converged to within a given tolerance $\textsc{tol} > 0$ or a maximum number of iterations is reached.
The most common test for convergence is when the norm of the residual,
\begin{equation}
\label{eq:res_k}
  \vec{r}^{(k)} \coloneqq  (A - \mu^{(k)}I)\vec{x}^{(k)},
\end{equation}
is below the tolerance $\textsc{tol}$ or the scaled tolerance $\mu^{(k)} \times \textsc{tol}$. The use of the residual as a convergence test is justified by the following \emph{a posteriori} bounds
on the errors.

\begin{proposition}[\cite{saad2011}, Cor. 3.4 and Thm. 3.9]
    Let $\vec{u} \in \bbC^n$ with $\norm{\vec{u}}=1$ and let $\mu = \raq_A(\vec{u})$ be the corresponding Rayleigh quotient. Let $(\lambda, \vec{v})$ be an eigenpair of $A$ such that $\lambda$ is the eigenvalue closest to $\mu$ and let $\delta > 0$ be the distance from $\lambda$ to the rest of the spectrum. Then the residual,
    $\vec{r} = (A - \mu I) \vec{u}$, satisfies
    \begin{equation*}
        |\mu - \lambda| \le \frac{\norm{\vec{r}}^2}{\delta}
        \qquad\text{and}\qquad
        \sin \theta(\vec{u}, \vec{v}) \le \frac{\norm{\vec{r}}}{\delta}\,.
    \end{equation*}
\end{proposition}

Although these bounds depend inversely on the eigenvalue gap, and therefore become inaccurate if the target eigenvalue $\lambda$ is not well-separated from its neighbours, a residual based stopping criterion is often the only option to determine convergence. 
Since RQI converges locally cubically (see below), it is usually sufficient to perform one additional iteration after the stopping criterion is satisfied to guarantee that the computed eigenvalue and eigenvector approximation are sufficiently accurate.

\begin{theorem}[\cite{boermmehl}, Proposition 4.13]%
  \label{thm:rqi:conv}
  For an initial vector $\vec{x}^{(0)} \in \bbC^n$, let ${(\vec{x}^{(k)})}_{k \ge 0}$ be the sequence of vectors generated by RQI and let $\mu^{(k)} = \raq_{A}(\vec{x}^{(k)})$ for $k \ge 0$. Assume that $A - \mu^{(0)}I$ is invertible and label the eigenvalues of $A$ such that
  \begin{equation}
  \label{eq:eigval_order:classic}
    \big|{\lambda_1 - \mu^{(0)}}\big| < \big|{\lambda_2 - \mu^{(0)}}\big| \le \cdots \le \big|{\lambda_n - \mu^{(0)}}\big|\,.
  \end{equation}
  Suppose that there exists a constant $c > 0$ such that
  \begin{equation}
    \label{eq:constant:rqi}
    \tan\big(\theta(\vec{x}^{(0)},\vec{v}_1)\big) \le c \le \sqrt{\frac{\abs{\lambda_1 - \lambda_2}}{2\norm{A - \lambda_1 I}}}\,,
  \end{equation}
  then the next RQI iterate $\vec{x}^{(1)}$ satisfies
  \begin{equation*}
    \tan\big(\theta(\vec{x}^{(1)}, \vec{v}_1) \big) \le 2 \frac{\norm{A - \lambda_1 I}}{\abs{\lambda_1 - \lambda_2}} \tan^3\big(\theta(\vec{x}^{(0)}, \vec{v}_1)\big) \leq c\,,
  \end{equation*}
  i.e., RQI is locally cubically convergent.
\end{theorem}

Although the condition in~\eqref{eq:constant:rqi} is only sufficient but not necessary to ensure convergence to the target eigenvector, it suggests that very accurate initial vectors $\vec{x}^{(0)}$ are required to ensure convergence to the target eigenpair when the gap $\abs{\lambda_1 - \lambda_2}$ is small. When the initial vector is not well-aligned with the target eigenspace, the convergence behaviour of RQI is difficult to predict. In~\cite{battersonsmilie}, the dependence of the basins of attraction around the eigenvectors on the eigenvalue gap is studied. For the $3\times 3$ case it is shown that the basins deteriorate when the eigenvalues around the target eigenvalue are closely spaced; analogous observations and visualisations can also be found in~\cite{pantazisszyld, absil}. In Section~\ref{sec:numerics}, we provide similar visualisations for RQI and our method, which demonstrate a significantly reduced dependence on the eigenvalue gap of our method. 

The idea of modifying RQI to make it more predictable is not new. However, most approaches are aimed at designing methods that ensure convergence to an eigenvalue near the initial shift or in a predefined interval~\cite{szyld, beattiefox}. These approaches are not suitable, however, for very closely spaced eigenvalues. Our goal is to modify RQI so that it utilises the information in the initial vector $\vec{x}^{(0)}$ rather than the initial shift. In this way, the convergence becomes more predictable and less dependent on the eigenvalue gap.

\section{Projected Rayleigh Quotient Iteration}\label{sec:prqi}
In this section, we present the new \textit{Projected Rayleigh Quotient Iteration} (PRQI), along with a simplified version of the algorithm and 
an analysis of the convergence properties.
First, we provide some intuition behind the main idea of introducing a complex
projection into the shift.

Suppose that $\vec{u} \in \bbC^n$ is a unit norm approximation for one of the eigenvectors of a Hermitian matrix $A \in \bbC^{n \times n}$. After relabelling we may assume that 
$\vec{u}$ approximates $\vec{v}_1$, which we shall
refer to as the \emph{target eigenvector}. 
Expanding $\vec{u}$ in the eigenbasis of $A$, we can write 
\[
\vec{u} = \sum_{j = 1}^n \alpha_j \vec{v}_j\,,
\]
where $\abs{\alpha_1} \approx 1$ and $\abs{\alpha_j} \ll 1$ for  $j > 1$,
because $\vec{u} \approx \vec{v}_1$. 
Now, consider the projection onto the orthogonal complement of the span of $\vec{u}$, given by the matrix $I - \vec{u}\vec{u}^\ast$. Adding this projection to $A$, with a complex scaling, for some eigenvector $\vec{v}_m$ we have
\begin{align*}
    \big[A + i(I - \vec{u}\vec{u}^\ast)\big]\vec{v}_m
    &= \bigg[A + i\bigg( I - \sum_{j,k = 1}^n \alpha_j \alpha_k^\ast \vec{v}_j \vec{v}_k^\ast\bigg)\bigg]\vec{v}_m \\
    &= \lambda_m\vec{v}_m + i\bigg( \vec{v}_m - \alpha_m^\ast \sum_{j = 2}^n \alpha_j\vec{v}_j - \alpha_m^\ast \alpha_1 \vec{v}_1\bigg)\,.
\end{align*}
If $\lambda_m$ is not the target eigenvalue, i.e.,  $m\neq 1$, then the last two terms inside the bracket are approximately zero and we have
\begin{equation*}
  \big[A + i(I - \vec{u}\vec{u}^\ast)\big]\vec{v}_m \approx (\lambda_m + i)\vec{v}_m, \quad
  \text{for }m > 1\,.
\end{equation*}
On the other hand, for $m = 1$, since $|\alpha_j| \ll 1$ for $j > 1$,
the sum is approximately zero and since $|\alpha_1| \approx 1$ we have
\begin{equation*}
  \big[A + i(I - \vec{u}\vec{u}^\ast)\big]\vec{v}_1  \approx \lambda_1\vec{v}_1 + i\vec{v}_1 -i\abs{\alpha_1}^2 \vec{v}_1 \approx \lambda_1 \vec{v}_1\,.
\end{equation*}
Hence, perturbing the matrix $A$ by $i(I - \vec{u}\vec{u}^\ast)$ we expect that the eigenvalue corresponding to $\vec{v}_1$ to be barely affected while the remaining eigenvalues will be shifted by approximately~$i$.

To control the magnitude of the perturbation, we introduce the (real) scalar $\gamma > 0$ and consider instead
\begin{equation}
  \label{eq:atilde}
  A \longrightarrow A + i\gamma(I - \vec{u}\vec{u}^\ast) =: \widetilde{A}\,.
\end{equation}
Denoting the eigenvalues of the perturbed matrix $\widetilde{A}$ by
$\widetilde{\lambda}_m$, following the arguments above we have
\begin{equation}
\label{eq:shift_evals}
  \widetilde{\lambda}_1 \approx \lambda_1 \qquad\text{and}\qquad \widetilde{\lambda}_m \approx \lambda_m + i\gamma \quad \text{for } m > 1\,,
\end{equation}
i.e., the eigenvalues of $A$ (except for the target eigenvalue $\lambda_1$) are raised into the complex plane and the distance between the target eigenvalue and the remaining eigenvalues is now approximately $\gamma$ {(again, see Figure~\ref{fig:perturbed_spectrum})}.

{By applying RQI to the matrix $\widetilde{A}$ defined in~\eqref{eq:atilde}, we would compute an eigenpair of $\widetilde{A}$ instead of $A$. Therefore, we replace the initial approximation $\vec{u}$ by the current vector iterate $\vec{x}^{(k)}$ and use the resulting matrix in a shift-invert method with the shift chosen as the Rayleigh quotient of $\vec{x}^{(k)}$ with respect to $A$. Since the approximations in~\eqref{eq:shift_evals} become exact if $\vec{u} = \vec{v}_1$, for sufficiently accurate initial shifts the method will converge to the correct eigenpair $(\lambda_1, \vec{v}_1)$. Due to the increased eigenvalue gap the basin of attraction around the target eigenvector increases, thus making the algorithm more likely to succeed. To further accelerate the method, we also replace the fixed scaling parameter $\gamma$ by a non-negative sequence ${(\gamma^{(k)})}_{k \ge 0}$ that converges to 0.}

The full iteration is summarised in Algorithm~\ref{alg:prqi1}. The stopping criteria can be chosen to be the same as for classic RQI, e.g., terminating when either the norm of the residual \eqref{eq:res_k} is below a given tolerance or a maximum number of iterations is reached.

For real symmetric problems, the eigenvectors are real but due to the imaginary shift, the final vector iterate $\vec{x}^{(K)}$ typically still contains a non-zero imaginary part. We therefore usually perform one step of classic RQI applied to $\real(\vec{x}^{(K)})$ after the loop ($\real(\cdot)$ denotes the component-wise real part).

\begin{algorithm}[t]
  \caption{Projected Rayleigh Quotient Iteration (full version)}%
  \label{alg:prqi1}
  \begin{algorithmic}[1]
    \Require $\vec{x}^{(0)} \in \bbC^n$ with $\|\vec{x}^{(0)}\|=1$
    \For{$k=0,1,2,\dotsc$, until stopping criteria are satisfied}
      \State
      $\mu^{(k)} \gets {(\vec{x}^{(k)})}^\ast A \vec{x}^{(k)}$
      \Comment Compute the Rayleigh Quotient

      \State Update $\gamma^{(k)}$ 
      \Comment{Update scaling factor}

      \State Solve
      $\big[A - \mu^{(k)}I + i\gamma^{(k)}(I -
      \vec{x}^{(k)}{(\vec{x}^{(k)})}^\ast)\big]\vec{y}^{(k+1)} =
      \vec{x}^{(k)}$ for $\vec{y}^{(k+1)}$

      \State
      $\vec{x}^{(k+1)} \gets \vec{y}^{(k+1)} / \|\vec{y}^{(k+1)}\|$
      \Comment{Normalise} \EndFor
  \end{algorithmic}
\end{algorithm}

Note that in practice we allow each $\gamma^{(k)}$ to depend on the previous iterates 
$\vec{x}^{(0)},$ $\vec{x}^{(1)}, \ldots,\, \vec{x}^{(k)}$. Hence, 
the sequence ${(\gamma^{(k)})}_{k \ge 0}$ is not an input to Algorithm~\ref{alg:prqi1},
but instead we have included Step 3 where $\gamma^{(k)}$ is updated.
How the sequence ${(\gamma^{(k)})}_{k \ge 0}$ should be chosen to ensure both stability of the iteration and rapid convergence is discussed in Section~\ref{sec:convergence}. 
First, we derive a simplified version of the algorithm.

\subsection{A simplified algorithm}
The solution of the linear system in Step 4 of Algorithm~\ref{alg:prqi1} requires the solution of the rank-one updated system
\begin{equation*}
  \big[\widetilde{A}^{(k)} - i
  \gamma^{(k)}\vec{x}^{(k)}{(\vec{x}^{(k)})}^\ast\big] \vec{y}^{(k+1)} = \vec{x}^{(k)}\,, 
\end{equation*}
where
\begin{equation*}
  \widetilde{A}^{(k)} \coloneqq A - \mu^{(k)}I + i\gamma^{(k)}I\,.
\end{equation*}
Such systems can be solved by first solving $\widetilde{A}^{(k)} \vec{z}^{(k)} = \vec{x}^{(k)}$ and then updating the solution using the Sherman--Morrison formula as described in, e.g., \cite{hager}. We now show that updating the solution is not necessary because solving $\widetilde{A}^{(k)} \vec{z}^{(k)} = \vec{x}^{(k)}$ and normalising the solution produces the same iterates as Algorithm~\ref{alg:prqi1}.

\begin{lemma}
\label{lem:simple_shift}
  Let $A \in \bbC^{n \times n}$ be Hermitian and $\vec{u} \in \bbC^{n}$ with $\|\vec{u}\| = 1$. Define the matrices\removespace
  \begin{equation*}
    B \coloneqq A - \mu I + i \gamma I
  \quad \text{and} \quad
    C \coloneqq { A - \mu I + i \gamma( I - \vec{u} \vec{u}^\ast) } \,,
  \end{equation*}
  where $\gamma > 0$ and $\mu \in \bbR$ is such that $(\vec{u}, \mu)$ is not an eigenpair of $A$. Then, both $B$ and $C$ are invertible and $B^{-1} \vec{u}$ is a scalar multiple of $C^{-1} \vec{u}$. 
\end{lemma}
\begin{proof}
  Writing $B = A - (\mu - i\gamma)I$, we see that $B$ is singular if and only if $\mu - i \gamma$ is an eigenvalue of $A$. Since $A$ is Hermitian all of its eigenvalues are real, but since $\gamma > 0$, we have $\mu - i \gamma \in \bbC \setminus \bbR$, which cannot be an eigenvalue of $A$. Thus, $B$ is invertible.
  
  Now, observe that $C = B - i \gamma \vec{u}\vec{u}^\ast$. The Sherman--Morrison formula (see, e.g., \cite[p.~50]{golub1996}) states that $C$ is invertible if $i \gamma \vec{u}^\ast B^{-1} \vec{u} \neq 1$. To check that this holds, let $A = V D V^\ast$ be an eigendecomposition of $A$ with a unitary matrix $V$, whose columns are the eigenvectors of $A$, and a diagonal matrix $D = \diag(\lambda_1, \dotsc, \lambda_n)$ containing the corresponding eigenvalues. We then have $B = V(D - \mu I + i \gamma I)V^\ast$ and consequently
  \begin{equation*}
      B^{-1} 
      =  {(V^\ast)}^{-1} {(D- \mu I + i \gamma I)}^{-1} V^{-1} 
      = V {(D- \mu I + i \gamma I)}^{-1} V^\ast \,.
  \end{equation*}
Thus,
  {\begin{equation*}
      \vec{u}^\ast B^{-1} \vec{u} 
      = \vec{u}^\ast V {(D- \mu I + i \gamma I)}^{-1} V^\ast \vec{u}
      = \sum_{j=1}^n \frac{\abs{w_j}^2}{\lambda_j - \mu + i\gamma}\,,
  \end{equation*}}
  where we introduced $\vec{w} \coloneqq V^\ast\vec{u}$. 
Multiplying by $i\gamma$ and taking the {imaginary} part gives
{\begin{equation*}
    \imag(i \gamma \vec{u}^\ast B^{-1} \vec{u})
     = \sum_{j=1}^n \frac{\gamma (\lambda_j - \mu) \abs{w_j}^2}{{(\lambda_j - \mu)}^2 + \gamma^2} \neq 0.
\end{equation*}}
In particular, $i\gamma \vec{u}^\ast B^{-1} \vec{u} \neq 1$ so that $C$ is also invertible. {Defining $\beta \coloneqq 1 - i \gamma \vec{u}^\ast B^{-1} \vec{u} \neq 0$, the Sherman--Morrison formula thus gives
\begin{align*}
    C^{-1} \vec{u} &= {\big( B - i \gamma \vec{u} \vec{u}^\ast \big)}^{-1} \vec{u} \\
                         &= \big(
                           B^{-1} + \beta^{-1} B^{-1} \vec{u} i \gamma \vec{u}^\ast B^{-1}
                           \big) \vec{u}
                         = B^{-1} \vec{u} ( 1 + \beta^{-1} i \gamma \vec{u}^\ast B^{-1} \vec{u})
                         \,.
\end{align*} 
}
\end{proof}

Replacing $\mu$ by $\mu^{(k)}$, $\gamma$ by $\gamma^{(k)}$ and $\vec{u}$ by $\vec{x}^{(k)}$ in Lemma~\ref{lem:simple_shift} implies that the solution $\vec{y}^{(k)}$ of the linear system in Step 4 of Algorithm~\ref{alg:prqi1} is a scalar multiple of the solution $\vec{z}^{(k)}$ of%
\removespace
\begin{equation*}
  \big[A - \mu^{(k)}I + i\gamma^{(k)}I\big]\vec{z}^{(k)} = \vec{x}^{(k)}\,.
\end{equation*}
Since the next iterate $\vec{x}^{(k+1)}$ is obtained by normalising $\vec{y}^{(k)}$, normalising $\vec{z}^{(k)}$ yields the same iterate (up to a sign).
Hence, it suffices to simply add an imaginary shift and Algorithm~\ref{alg:prqi1} can be simplified to the iteration given in Algorithm~\ref{alg:prqi2}.

\begin{algorithm}[t]
  \caption{Projected Rayleigh Quotient Iteration (simplified version)}%
  \label{alg:prqi2}
  \begin{algorithmic}
    \Require $\vec{x}^{(0)} \in \bbC^n$ with $\|\vec{x}^{(0)}\|=1$
    \For{$k=0,1,2,\dotsc$, until stopping criteria are satisfied}
    \State
    $\mu^{(k)} \gets {(\vec{x}^{(k)})}^\ast A \vec{x}^{(k)}$
    \Comment Compute the Rayleigh Quotient

    \State Update $\gamma^{(k)}$
    \Comment{Update scaling factor}

    \State Solve
    $\big[A - (\mu^{(k)} - i\gamma^{(k)})I\big]\vec{z}^{(k+1)} =
    \vec{x}^{(k)}$ for $\vec{z}^{(k+1)}$

    \State
    $\vec{x}^{(k+1)} \gets \vec{z}^{(k+1)} / \|{\vec{z}^{(k+1)}}\|$
    \Comment{Normalise} \EndFor
  \end{algorithmic}
\end{algorithm}

\subsection{Convergence}\label{sec:convergence}
We now turn to the question of how to choose ${(\gamma^{(k)})}_{k \ge 0}$. Throughout this section, $(\lambda_1, \vec{v}_1)$ denotes the target eigenpair of $A$ and ${(\vec{x}^{(k)})}_{k \ge 0}$ is the sequence of vectors that is obtained by Algorithm~\ref{alg:prqi2}.

The rapid convergence rate of RQI can be partly attributed to quadratic accuracy of the Rayleigh Quotient (cf.\ Theorem~\ref{thm:quadacc}), i.e.,
\begin{equation*}
  | \mu^{(k)} - \lambda_1| =
  \mathcal{O}\big(\tan^2(\theta(\vec{x}^{(k)}, \vec{v}_1)) \big)\,,
\end{equation*}
which leads to cubic convergence as in Theorem~\ref{thm:rqi:conv}.
In light of this, we will assume that $\gamma^{(k)}$ is chosen to decay either linearly or quadratically in $\tan(\theta(\vec{x}^{(k)}, \vec{v}_1)$.
{The following condition initially may seem restrictive, but it can be
satisfied by the natural choice of defining $\gamma^{(k)}$ in terms of the residual
(see Lemma~\ref{lem:res_spread} below).}

\begin{assumption}
\label{asm:gamma}
For $q \in \{ 0,1\}$ let ${(\gamma^{(k)})}_{k \ge 0}$ be such that 
\begin{equation}
\label{eq:gamma error bound}
  0 < \gamma^{(k)} \leq
  C_\gamma \tan^{1 + q}(\theta(\vec{x}^{(k)}, \vec{v}_1))
  \quad \text{for all } k \ge 0 \ \text{and some } C_\gamma > 0.
\end{equation}
\end{assumption}

As we will show in the following theorem, this results
in an asymptotic convergence rate that is either quadratic (if $q = 0$ in \eqref{eq:gamma error bound}) or cubic (if $q = 1$ in \eqref{eq:gamma error bound}). In practice, we have observed that choosing $\gamma^{(k)}$ to obtain quadratic convergence ($q = 0$ in \eqref{eq:gamma error bound}) usually only leads to a small number of extra iterations to achieve the desired error tolerance, but has the benefit
of adding stability to the method. {Note that the theorem uses the same eigenvalue
ordering as Theorem~\ref{thm:ii:convergence} (with $\mu = \mu^{(0)} = \raq_A(\vec{x}^{(0)})$) and also as \eqref{eq:eigval_order:classic} in Theorem~\ref{thm:rqi:conv}.}

\begin{theorem}\label{thm:prqi:conv}
   For an initial vector $\vec{x}^{(0)} \in \bbC^n$, let ${(\vec{x}^{(k)})}_{k \ge 0}$ be the sequence of vectors generated by PRQI as in Algorithm~\ref{alg:prqi2} with $(\gamma^{(k)})_{k \ge 0}$ satisfying Assumption~\ref{asm:gamma} with $q = 0$ or $1$, and let $\mu^{(k)} = \raq_{A}(\vec{x}^{(k)})$ for $k \ge 0$. Label the eigenvalues of $A$ such that
  \begin{equation}
    \label{eq:eigval:order}
    |\lambda_1 - \mu^{(0)}| < |\lambda_2 - \mu^{(0)}| \le \cdots \le |\lambda_n - \mu^{(0)}|\,.
  \end{equation}
  Suppose that there exists a constant $c > 0$ such that 
  \begin{equation}
    \label{eq:prqi:const}
    \tan(\theta(\vec{x}^{(0)},\vec{v}_1)) \le c \le {\left({\frac{\abs{\lambda_1 - \lambda_2}}{2(\norm{A - \lambda_1 I} + C_\gamma)}}\right)}^{1/(1+q)}.
  \end{equation}
  for $C_\gamma > 0$ as in \eqref{eq:gamma error bound}, then the next PRQI iterate $\vec{x}^{(1)}$ satisfies 
  \begin{equation}
\label{eq:prqi_err}
    \tan(\theta(\vec{x}^{(1)}, \vec{v}_1)) \le 2 \frac{\norm{A - \lambda_1 I} +  C_\gamma}{\abs{\lambda_1 - \lambda_2}} \tan^{2 + q}(\theta(\vec{x}^{(0)}, \vec{v}_1)) \le c\,,
  \end{equation}
  i.e., PRQI is locally quadratically or cubically convergent,
  depending on $(\gamma^{(k)})_{k \ge 0}$.
\end{theorem}

{\begin{proof}
  We interpret one step of Algorithm~\ref{alg:prqi2} as one step of Inverse Iteration applied to the matrix $A$ using the shift $\sigma^{(0)} \coloneqq \mu^{(0)} - i\gamma^{(0)}$. Since 
  $A$ is Hermitian we have $\mu^{(0)} = \mathcal{R}_{A}(\vec{x}^{(0)}) \in \bbR$,
  and thus Equation~\eqref{eq:eigval:order} implies that
  \begin{equation*}
    |\lambda_1 - \sigma^{(0)}| < |\lambda_2 - \sigma^{(0)}| \le \cdots \le |\lambda_n - \sigma^{(0)}|\,,
  \end{equation*}
  so that the convergence theorem of Inverse Iteration (Theorem~\ref{thm:ii:convergence}) can be applied to obtain
  \begin{equation}
    \label{eq:sii:convprqi}
    \tan(\theta(\vec{x}^{(1)},\vec{v}_1)) \le \frac{|\lambda_1 - \sigma^{(0)}|}{|\lambda_2 - \sigma^{(0)}|} \tan(\theta(\vec{x}^{(0)}, \vec{v}_1))\,.
  \end{equation}
  
  Let us first bound the numerator in the right hand side. By using the quadratic accuracy of the Rayleigh Quotient \eqref{eq:RQ_quad} to bound $|\lambda_1 - \mu^{(0)}| $, recalling that 
  $\mu^{(0)} = \mathcal{R}_A(\vec{x}^{(0)})$, and Assumption~\ref{asm:gamma} to bound $\gamma^{(0)}$ we have
  \begin{align}
  \label{eq:lam1-sig0}
    |\lambda_1 - \sigma^{(0)}|
    \le (\|A - \lambda_1I\| + C_\gamma) \tan^{1 + q}(\theta(\vec{x}^{(0)}, \vec{v}_1))\,.
  \end{align}

Expanding the denominator and using \eqref{eq:lam1-sig0} as well as \eqref{eq:prqi:const}, we obtain the bound
  \begin{align}
  \label{eq:lam2-sig0}
    |\lambda_2 - \sigma^{(0)}|
    &\ge |\lambda_1 - \lambda_2| - |\lambda_1 - \sigma^{(0)}| \\
    &\ge |\lambda_1 - \lambda_2| - \bigl(\|A - \lambda_1I\| + C_\gamma\bigr) \tan^{1+q}(\theta(\vec{x}^{(0)}, \vec{v}_1)) 
    \ge    \frac{|\lambda_1 - \lambda_2|}{2}\,.\nonumber
  \end{align}

  Inserting the estimates \eqref{eq:lam1-sig0} and \eqref{eq:lam2-sig0} into~\eqref{eq:sii:convprqi} gives the desired result \eqref{eq:prqi_err},
  \begin{equation*}
      \tan(\theta(\vec{x}^{(1)},\vec{v}_1))
    \le 2\frac{\norm{A - \lambda_1 I} + C_\gamma}{\abs{\lambda_1 - \lambda_2}}\tan^{2 + q}(\theta(\vec{x}^{(0)}, \vec{v}_1))\,.
  \end{equation*}
  In addition, using twice \eqref{eq:prqi:const} we also obtain
  \begin{align*}
    \tan(\theta(\vec{x}^{(1)},\vec{v}_1))
    \le 2\frac{\norm{A - \lambda_1 I} + C_\gamma}{\abs{\lambda_1 - \lambda_2}} \frac{\abs{\lambda_1 - \lambda_2}}{2(\norm{A - \lambda_1 I} + C_\gamma)}\tan(\theta(\vec{x}^{(0)}, \vec{v}_1))
    \le c\,.
  \end{align*}
\end{proof}
}

Let us compare the conditions~\eqref{eq:eigval:order} and~\eqref{eq:prqi:const} with the corresponding conditions of classic RQI, i.e.,~\eqref{eq:eigval_order:classic} and~\eqref{eq:constant:rqi}. The conditions~\eqref{eq:eigval_order:classic} and~\eqref{eq:eigval:order} are the same and state that the initial shift (i.e., the Rayleigh Quotient of the initial vector) must be closer to $\lambda_1$ (the target eigenvalue) than to any other eigenvalue.

To compare \eqref{eq:constant:rqi} and \eqref{eq:prqi:const}, note that
\begin{equation*}
    \sqrt{\frac{\abs{\lambda_1 - \lambda_2}}{2\norm{A - \lambda_1 I}}} >
    {\left(\frac{\abs{\lambda_1 - \lambda_2}}{2(\norm{A - \lambda_1 I} + C_\gamma)}\right)}^{1/(1+q)}\,,
\end{equation*}
for both $q = 0$ and $1$. The left hand side of this inequality is the upper bound that appears in condition~\eqref{eq:constant:rqi} of classic RQI while the term on the right hand side is precisely the one in~\eqref{eq:prqi:const} in the convergence theorem of PRQI. Thus, the conditions in our convergence theorem are stronger than the respective conditions of classic RQI. In other words, the convergence result given above does not show that our method allows for less accurate initial vectors compared to classic RQI. 

This is not very surprising since we essentially apply RQI to the matrix $A$ but then use the ``wrong'' matrix $A - i\gamma^{(k)}I$ when solving the linear system during the iteration. We believe that in order to show that our method does allow for less accurate initial guesses, the original formulation as given in Algorithm~\ref{alg:prqi1} has to be considered and the effect of the perturbation (cf.~\eqref{eq:atilde}) on the spectrum needs to be quantified,
{which could be an interesting direction for future research.
Another plausible approach is to interpret one step of PRQI as one step of an inexact inverse iteration, as in, e.g., \cite{BERNSMULLER2006389}, applied to the matrix $\widehat{A}^{(k)} \coloneqq A + i\gamma^{(k)}(I - \vec{v}_1 \vec{v}_1^\ast)$ with shift $\mu^{(k)} = \raq_A(\vec{x}^{(k)})$ and where $\vec{v}_1 \in \bbC^n$ is the target eigenvector.
However, this method of proof again leads to \eqref{eq:sii:convprqi} and so does not 
offer any improvement over the result above.
}

Let us now turn to the question of how to choose the sequence ${(\gamma^{(k)})}_{k \ge 0}$ in practice. An easily computable (and in most cases already available) choice for $\gamma^{(k)}$ is the norm (or the squared norm) of the current residual $\vec{r}^{(k)}$ as in \eqref{eq:res_k}.
The following Lemma justifies this choice by showing that the residual can be bounded from above in terms of the angle with the target eigenvector. Here, $\spread(A)$ denotes the spread of the eigenvalues of $A$, i.e., $\spread(A) = \lambda_{\max} - \lambda_{\min}$, where $\lambda_{\max}$ and $\lambda_{\min}$ are the largest and smallest eigenvalue of $A$, respectively.

\begin{lemma}
  \label{lem:res_spread}
  Let $\vec{u} \in \bbC^{n}$, with $\|\vec{u}\| = 1$, be such that
  $\mu = \vec{u}^\ast A \vec{u}$ is closer to $\lambda_1$ than to any
  other eigenvalue of $A$. Then
  \begin{equation*}
    \|(A - \mu I)\vec{u}\|\le \spread(A) \tan(\theta(\vec{u},\vec{v}_1)) \,.
  \end{equation*}
\end{lemma}
\begin{proof}
Theorem 11.7.1 in \cite{parlett1998} gives
\[
\|(A - \mu I)\vec{u}\| \le \spread(A) \sin(\theta(\vec{u},\vec{v}_1))\,.
\]
The fact that $\sin(\theta(\vec{x}^{(0)}, \vec{v}_1)) < \tan(\theta(\vec{x}^{(0)}, \vec{v}_1))$ follows from the definition of the angle (Equation~\eqref{eq:angle_def}) and the Cauchy--Schwarz inequality.
\end{proof}

\begin{corollary}
If $\gamma^{(k)} = \|\vec{r}^{(k)}\|$ for all $k \ge 0$, 
then PRQI converges locally quadratically. 
If $\gamma^{(k)} = \|\vec{r}^{(k)}\|^2$ for all $k \ge 0$, 
then PRQI converges locally cubically. 
\end{corollary}

\subsection{Generalised Eigenvalue Problems}
Many practical applications involve \emph{generalised eigenvalue problems}, which,
for $M \in \bbC^{n \times n}$, are of the form
\begin{equation}
\label{eq:gen-evp}
    A\vec{v} = \lambda M \vec{v}\,.
\end{equation}
Such problems arise, e.g., when discretising a Sturm--Liouville problem using finite elements (cf.~Example~\ref{ex:sturm} in Section \ref{sec:numerics}). Thus, we also extend the PRQI to generalised eigenvalue problems for the case where $M$ is Hermitian and positive definite. 

All eigenvalues of the generalised problem \eqref{eq:gen-evp} are real and the corresponding eigenvectors $\vec{v}_j$, $j=1,\dotsc,n$, can be chosen to be $M$-orthonormal, i.e., $\vec{v}_j^\ast M \vec{v}_k = \delta_{jk}$, see, e.g., \cite[Theorem 15.3.3]{parlett1998}. 

The classic RQI for \eqref{eq:gen-evp} is obtained by replacing the Rayleigh quotient with the generalised Rayleigh quotient
\begin{equation}
\label{eq:gen-RQ}
    \raq_{A,M}(\vec{x}) = \frac{\vec{x}^\ast A \vec{x}}{\vec{x}^\ast M \vec{x}}\,,
\end{equation}
and solving the linear system 
\begin{equation}
\label{eq:gen-rqi-solve}
    \big[A - \raq_{A,M}(\vec{x}^{(k)}) M\big] \vec{y}^{(k+1)} = M\vec{x}^{(k)}
\end{equation}
for $\vec{y}^{(k+1)}$ at each step $k=0,1,\dotsc$, see~\cite[p.\,358]{parlett1998} for details. The procedure for the generalised RQI is then given by modifying Algorithm~\ref{alg:rqi} by replacing the linear solve in Step 3 with \eqref{eq:gen-rqi-solve} and instead normalising with respect to $\|\vec{x}\|_M \coloneqq \sqrt{\vec{x}^\ast M\vec{x}}$.
Convergence can again be tested using the norm of the residual, 
which is now given by $\vec{r}^{(k)}_M \coloneqq (A - \mu^{(k)}M)\vec{x}^{(k)}$.

Given $\vec{u} \approx \vec{v}$ a generalised PRQI for 
\eqref{eq:gen-evp} based on the perturbation
\begin{equation}
\label{eq:gen_pert1}
A \longrightarrow A + i \gamma (M - (M\vec{u})(M\vec{u})^\ast)
\end{equation}
can be derived similarly. The procedure for the generalised PRQI using \eqref{eq:gen_pert1}
is given by updating Algorithm~\ref{alg:prqi1} by replacing the linear solve in Step 4 with
\[
\big[A - \mu^{(k)}M + i \gamma^{(k)}(M - (M\vec{x}^{(k)})(M\vec{x}^{(k)})^*)\big]
M\vec{y}^{(k + 1)} = M\vec{x}^{(k)}
\]
and normalising with respect to $\|\cdot\|_M$. Since only two steps change, we omit the full details and instead 
present a simplified version of the generalised PRQI.

To this end, we can alternatively obtain a 
generalised PRQI by noting that \eqref{eq:gen-evp} can 
be rewritten as a standard Hermitian eigenvalue problem,
\begin{equation*}
    M^{-1/2} A M^{-1/2} \vec{w} = \lambda \vec{w}\,,
\end{equation*}
where $M^{-1/2} = {(M^{1/2})}^{-1}$, with $M^{1/2}$ denoting the matrix square root of $M$ (which exists since we assumed $M$ to be positive definite, see, e.g.,~\cite{horn_johnson_2012}), and $\vec{w} = M^{1/2} \vec{v}$. Then, as for the standard problem, the PRQI for \eqref{eq:gen-evp} is based on the perturbation
\begin{equation}
\label{eq:gen_pert2}
    M^{-1/2} A M^{-1/2} \longrightarrow M^{-1/2} A M^{-1/2} + i \gamma(I - \widetilde{\vec{w}} \widetilde{\vec{w}}^\ast )\,,
\end{equation}
where now $\widetilde{\vec{w}} \approx \vec{w} = M^{1/2}\vec{v}$.
Letting $\widetilde{\vec{w}} = M^{1/2}\vec{u}$, with $\vec{u} \approx \vec{v}$,
then multiplying \eqref{eq:gen_pert2} on the left and right by $M^{1/2}$ we
see that \eqref{eq:gen_pert2} is equivalent to \eqref{eq:gen_pert1}.

The form of the perturbation in \eqref{eq:gen_pert2} is useful because 
Lemma~\ref{lem:simple_shift} again implies that we can use a simplified PRQI algorithm similar to Algorithm~\ref{alg:prqi2}, 
which instead only requires solving the linear system
\begin{equation*}
    \big[M^{-1/2} A M^{-1/2} - (\mu^{(k)} - i \gamma^{(k)})\big] 
    M^{1/2}\vec{z}^{(k+1)} = M^{1/2} \vec{x}^{(k)}
\end{equation*}
for $\vec{z}^{(k+1)}$ at each step. Multiplying this equation by $M^{1/2}$ from the left we see that the computation of $M^{1/2}$ and $M^{-1/2}$ can be avoided,
and thus in the generalised PRQI it is sufficient to solve the linear system
\begin{equation*}
    \big[A - (\mu^{(k)} - i \gamma^{(k)})M\big] \vec{z}^{(k+1)} = M \vec{x}^{(k)}
\end{equation*}
for $\vec{z}^{(k)}$ at each step.
The simplified PRQI for a generalised eigenvalue problem based on this 
procedure is outlined in Algorithm~\ref{alg:prqi3}.

\begin{algorithm}[t]
  \caption{PRQI for a generalised eigenvalue problem}%
  \label{alg:prqi3}
  \begin{algorithmic}[1]
    \Require $\vec{x}^{(0)} \in \bbC^n$ with ${(\vec{x}^{(0)})}^\ast M \vec{x}^{(0)}=1$
    \For{$k=0,1,2,\dotsc$, until stopping criteria are satisfied}
    \State
    $\mu^{(k)} \gets {(\vec{x}^{(k)})}^\ast A \vec{x}^{(k)}$
    \Comment Compute (generalised) Rayleigh Quotient

    \State Update $\gamma^{(k)}$
    \Comment{Update scaling factor}

    \State Solve
    $\big[A - (\mu^{(k)} - i\gamma^{(k)})M\big]\vec{z}^{(k+1)} =
    M\vec{x}^{(k)}$ for $\vec{z}^{(k+1)}$

    \State
    $\vec{x}^{(k+1)} \gets \vec{z}^{(k+1)} / \sqrt{{(\vec{z}^{(k+1)})}^\ast M {\vec{z}^{(k+1)}}}$
    \Comment{Normalise} \EndFor
  \end{algorithmic}
\end{algorithm}

Since Algorithm~\ref{alg:prqi3} is equivalent to Algorithm~\ref{alg:prqi2} applied to $M^{-1/2} A M^{-1/2}$, our convergence analysis from Section~\ref{sec:convergence} still applies. Although the two algorithms are equivalent, in practice Algorithm~\ref{alg:prqi3} is preferable because it avoids computing $M^{1/2}$ and $M^{-1/2}$.
This is especially the case if $M$ is large or numerically ill conditioned.

\section{Numerical Experiments}\label{sec:numerics}
In this section, we present numerical results for six different test problems structured in three examples: the first example discusses the case of real symmetric $3\times 3$ matrices; the second example comprises a comparison of classic RQI and PRQI for four different test matrices; and the last example 
comes from a practical model for photonic crystal fibres.
All of the examples use the simplified PRQI outlined in Algorithm~\ref{alg:prqi2}. {The first two examples use} $\gamma^{(k)} = \| \vec{r}^{(k)} \|$ for $k=0,1,\dotsc$. This is mainly motivated by the first example since in this case PRQI is invariant under scaling and shifting of the matrix under consideration (see Proposition \ref{prop:prqi_invariance} below), which allows for a direct comparison with classic RQI. {The last example uses the squared residual norm for the imaginary shift, i.e., $\gamma^{(k)} = \| \vec{r}^{(k)} \|^2$.} We use a residual-based stopping criterion everywhere, i.e., the iteration is stopped when the norm of the residual falls below
$\textsc{tol} = 10^{-11}$ in Example~\ref{ex:comparison3x3}, below $\textsc{tol} = 10^{-15}$ in Example~\ref{ex:comparison2} and below $\textsc{tol} = 
10^{-8}$ in Example~\ref{ex:sturm}. The code that was used to generate the results is available online at 
\texttt{\url{https://github.com/nilsfriess/PRQI-Examples}}\;.

\begin{example}\label{ex:comparison3x3}
In~\cite{pantazisszyld}, the authors explain that for classic RQI it suffices to consider matrices $A = \diag(-1,s,1)$ for $s \in (-1,1)$ to characterise the convergence behaviour for the real symmetric $3 \times 3$-case completely. The argument is based on the fact that classic RQI is scale and shift invariant
\cite[Proposition 3.1]{pantazisszyld}. This also holds for our method
{in the case where $\gamma^{(k)}$ is given by norm of the residual $\vec{r}^{(k)}$. Note that it does not hold for squared residual norm, since it is not invariant under linear
transformations as in the proposition below.}

\begin{proposition}\label{prop:prqi_invariance}
Let $A \in \mathbb{R}^{n \times n}$ be real symmetric and let $B = \alpha A + \beta I$, $\alpha \neq 0$, $\beta \in \mathbb{R}$. Then, {for $\gamma^{(k)} = \|\vec{r}^{(k)}\|$,} PRQI applied to $A$ and PRQI applied to $B$ produce the same iterates $\vec{x}^{(k)}$, $k = 1,2,\dotsc,$ when started with the same vector $\vec{x}^{(0)}$.
\end{proposition}
\begin{proof}
    Let $k \in \mathbb{N}$. Define $\vec{y}_A, \vec{y}_B \in \mathbb{R}^n$ by
    \begin{equation}
    \label{eq:systems_A_B}
        \big[A - (\raq_A(\vec{x}^{(k)}) - i\gamma_A^{(k)})I\big]\vec{y}_A = \vec{x}^{(k)}
        \ \ \text{and}\ \ 
        \big[B - (\raq_B(\vec{x}^{(k)}) - i\gamma_B^{(k)})I\big]\vec{y}_B = \vec{x}^{(k)}
    \end{equation}
    where the subscripts 
    emphasise that we consider the residual w.r.t.\ the matrix $A$ or $B$, respectively, i.e., $\gamma^{(k)}_A = \|[A - \raq_A(\vec{x}^{(k)})I]\vec{x}^{(k)}\|$ and $\gamma^{(k)}_B$ is defined accordingly.
    
    Observe that
    \begin{align*}
B - \raq_B(\vec{x}^{(k)})I
        &= \alpha A + \beta I - \alpha\raq_A(\vec{x}^{(k)})I - \beta I
        = \alpha \big[A - \raq_A(\vec{x}^{(k)})I\big]
    \end{align*}
and hence, $\gamma^{(k)}_B = \alpha\gamma^{(k)}_A$.
    Then equating the left hand sides of both equations in~\eqref{eq:systems_A_B} implies that $\vec{y}_A = \alpha\vec{y}_B$, and since $\vec{y}_A$ and $\vec{y}_B$ are 
    the iterates of PRQI before normalisation, the claim follows.
\end{proof}

Let now $A \in \mathbb{R}^{3 \times 3}$ be a symmetric matrix with eigenvalues $\lambda_1 < \lambda_2 < \lambda_3$. We can scale the matrix such that $\lambda_3 - \lambda_1 = 2$, then we can shift the spectrum such that $\lambda_1 = -1, \lambda_2 \in (-1,1)$ and $\lambda_3 = 1$. Proposition~\ref{prop:prqi_invariance} ensures that these operations do not change the convergence behaviour of our method. Lastly, we can assume that the matrix $A$ is diagonal, because our method is independent under orthogonal transformations of the matrix $A$ (provided the initial vector is transformed accordingly). This is a well-known property of classic RQI and 
as above, can be verified for PRQI. In summary, for PRQI it is also sufficient to consider $A = \diag(-1, s, 1)$ for $s \in (-1,1)$.

\begin{figure}
  \centering
  \begin{subfigure}[b]{0.48\textwidth}
    \centering
    \includegraphics[width=\textwidth]{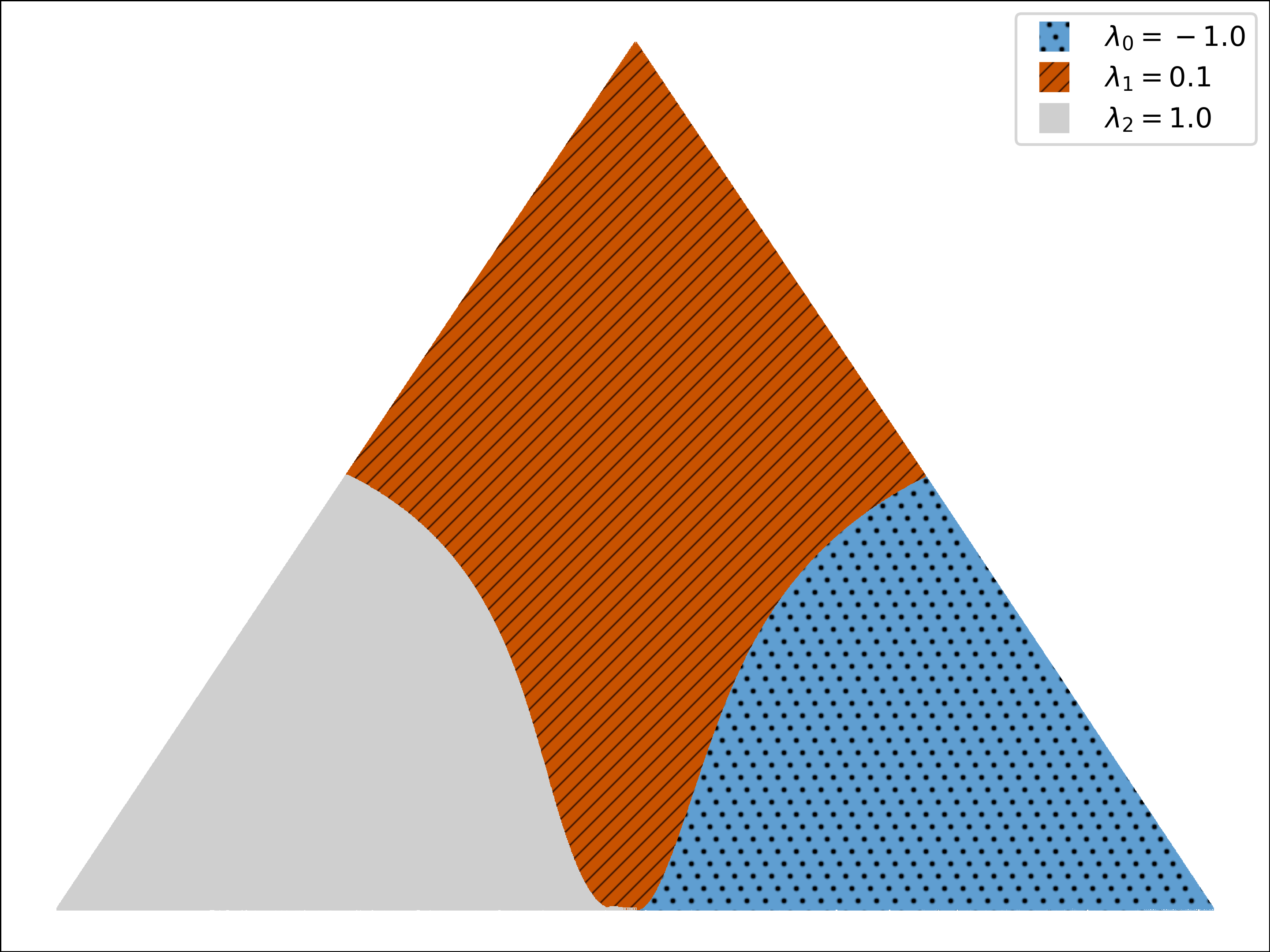}
    \caption{{\small Classic RQI, large gap}}%
    \label{fig:classic_rqi_large}
  \end{subfigure}
  \hfill
  \begin{subfigure}[b]{0.48\textwidth}
    \centering
    \includegraphics[width=\textwidth]{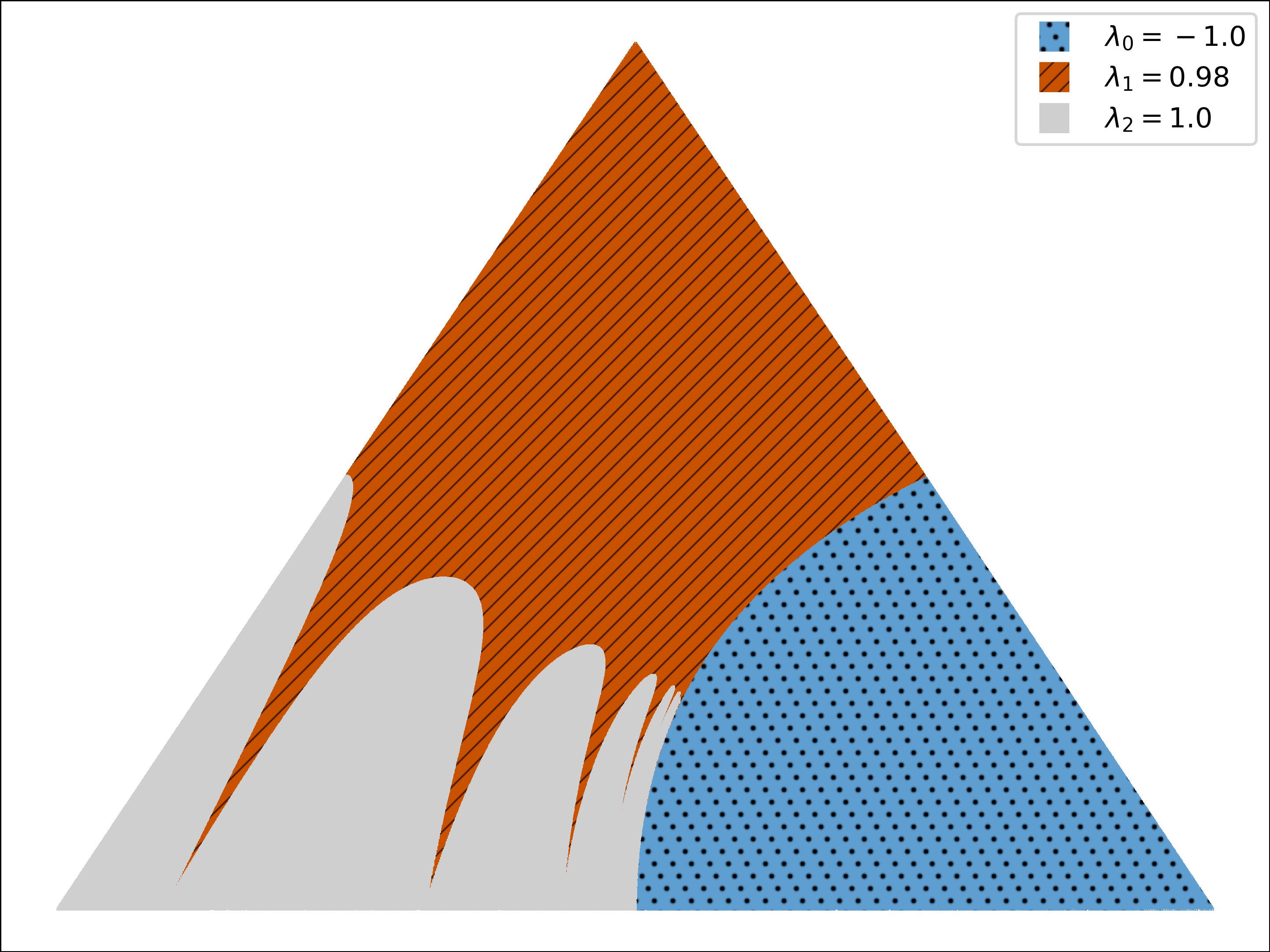}
    \caption{{\small Classic RQI, small gap}}%
    \label{fig:classic_rqi_small}
  \end{subfigure}
  \begin{subfigure}[b]{0.48\textwidth}
    \centering
    \includegraphics[width=\textwidth]{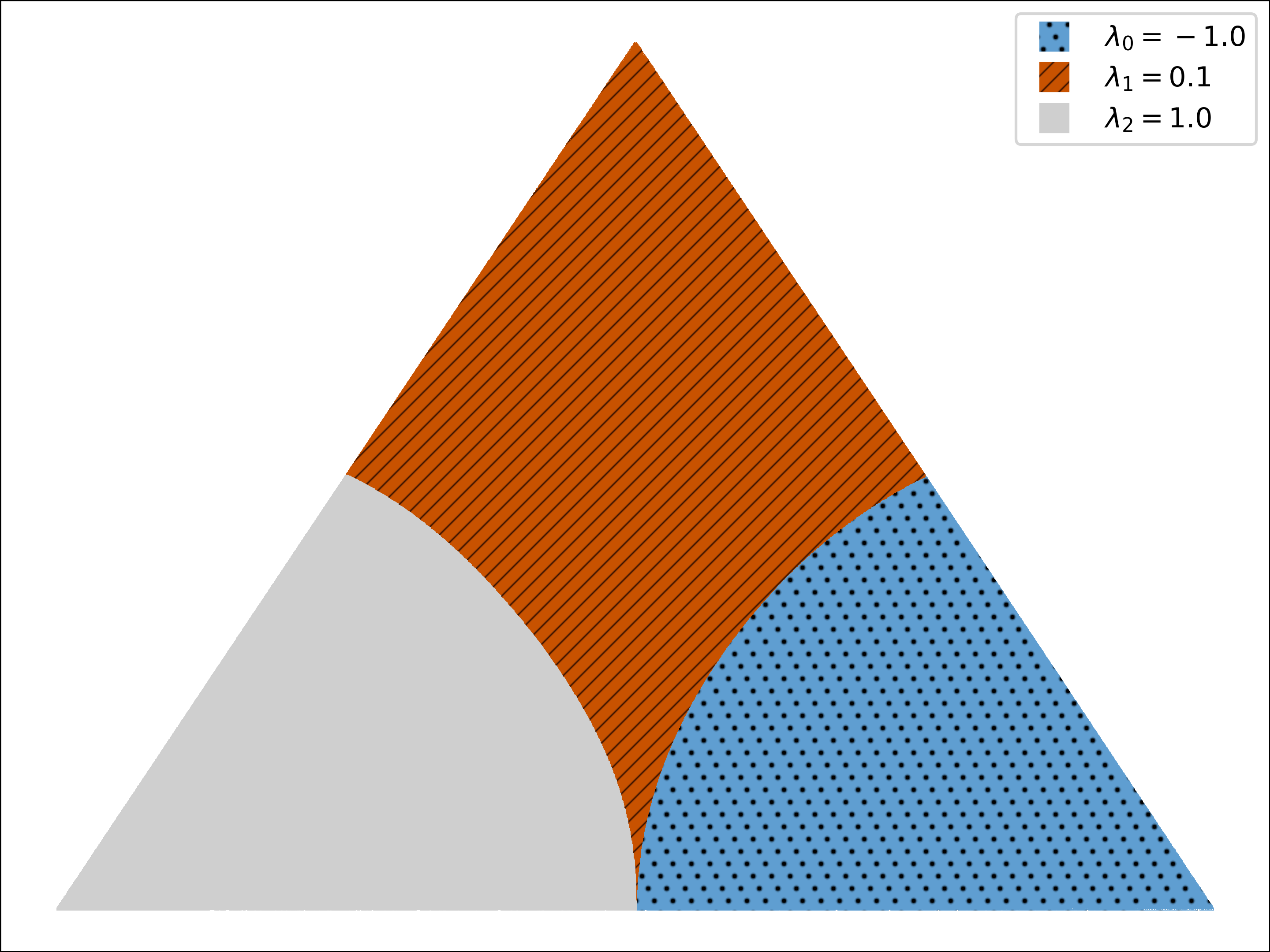}
    \caption{{\small Projected RQI, large gap}}%
    \label{fig:complex_rqi_large}
  \end{subfigure}
  \hfill
  \begin{subfigure}[b]{0.48\textwidth}
    \centering
    \includegraphics[width=\textwidth]{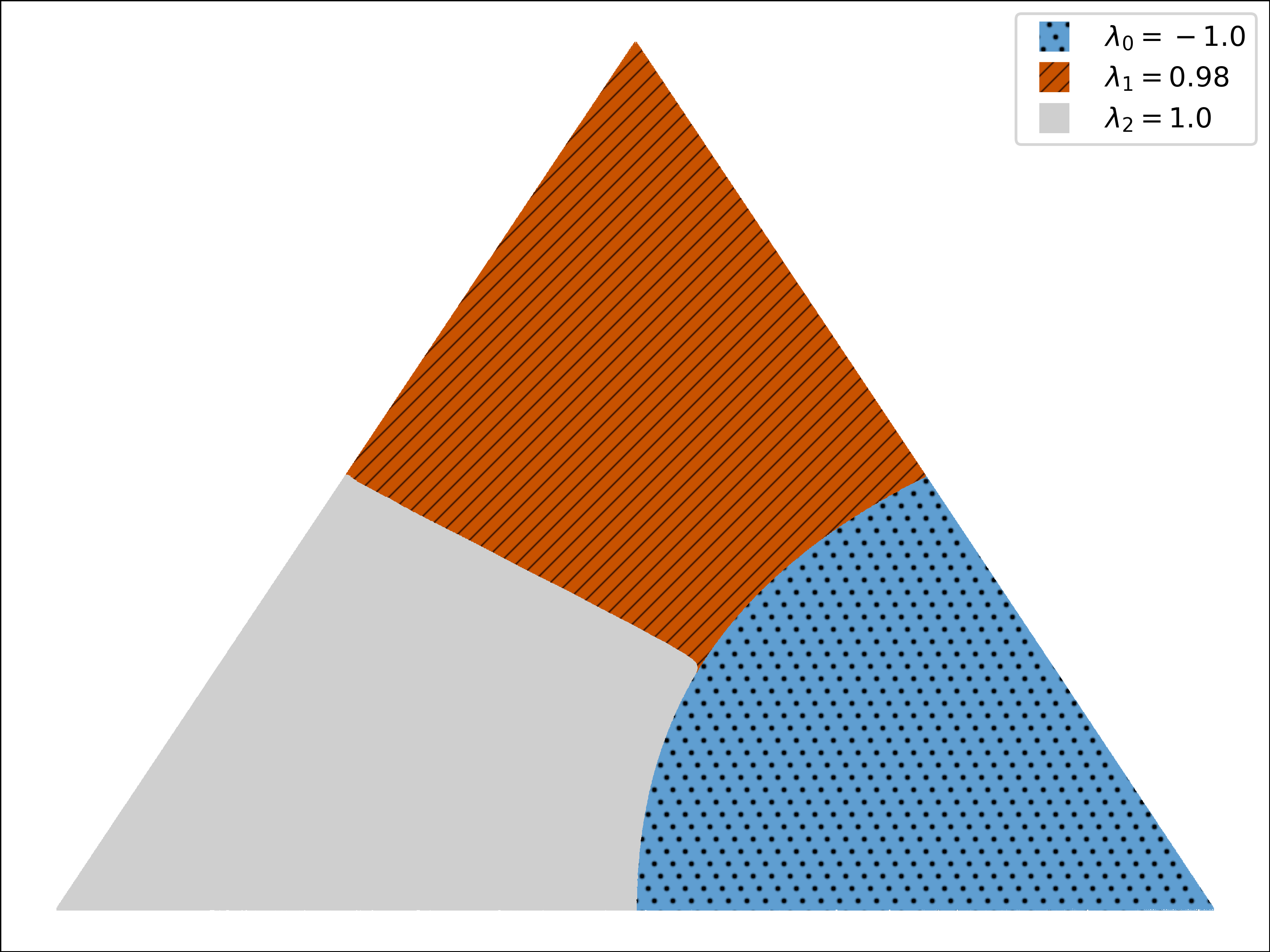}
    \caption{{\small Projected RQI, small gap}}%
    \label{fig:complex_rqi_small}
  \end{subfigure}
  \caption{\label{fig:rqi_shift_comparison} Visualisation of the basins of attraction for classic RQI (top row) and PRQI (bottom row) for different eigenvalue gaps in Example \ref{ex:comparison3x3}. If the eigenvalue gap is small (right column), the borders between the basins of attraction deteriorate for classic RQI; for PRQI the border is more regular.}%
\end{figure}
In Figure~\ref{fig:rqi_shift_comparison}, the basins of attraction of classic RQI and PRQI are visualised for $s = 0.1$ (left column) and $s = 0.98$ (right column). The figures were created by running both algorithms for many initial vectors and recording which eigenpair of $A$ they converged to. The three different coloured regions correspond to the three different eigenpairs. The plots are the intersection of the one-dimensional subspaces spanned by the chosen initial vectors and the unit simplex $\{(x_1,x_2,x_3) \in \bbR^3 \mid x_1 + x_2 + x_3 = 1, x_i > 0\}$ (similar Figures appeared in~\cite{battersonsmilie,absil}).

In the first row we observe the well-known behaviour of classic RQI: as the eigenvalue gap gets smaller, the basins of attraction deteriorate. Our method (second row) is more stable and less dependent on the spacing between eigenvalues. In particular, the border between the basins of attraction corresponding to the two closely spaced eigenvalues is more regular, 
which in practice results in a much more predictable convergence behaviour.
\end{example}

\begin{example}\label{ex:comparison2}
Comparing to classic RQI again, we demonstrate in this example for four test matrices that our PRQI method always converges to the correct eigenpair, provided the initial approximation is sufficiently accurate. 
The test matrices are:
\begin{enumerate}
    \item The $[1,2,1]$ matrix. This matrix is tridiagonal with all diagonal elements set to $2$ and all off-diagonal elements set to $1$.
    \item Wilkinson's $W_{2n+1}^+$ matrix. The $m$th diagonal entry of this tridiagonal matrix is set to $|n + 1 - m|$ and all off-diagonal entries are set to $1$. This matrix has pairs of very closely spaced eigenvalues, see~\cite[p.\,309]{wilkinson}.
    \item The Laplace matrix, which arises when discretising the Laplacian in 2D using a 5-point stencil with the finite difference method. Let $T$ be a $m \times m$ tridiagonal matrix with $4$ as its diagonal entries and $-1$ on its off-diagonals. The Laplace matrix is a $m^2 \times m^2$ block tridiagonal matrix with $T$ as its diagonal blocks and $-I$ as its off-diagonal blocks, where $I$ is the $m \times m$ identity matrix.        
    \item Symmetric random matrices with a random (symmetric) sparsity pattern and non-zero entries that are normally distributed with mean $0$ and variance~$1$.
\end{enumerate}
\begin{figure}
  \centering
  \begin{subfigure}[b]{0.45\textwidth}
    \centering
    \includegraphics[width=\textwidth]{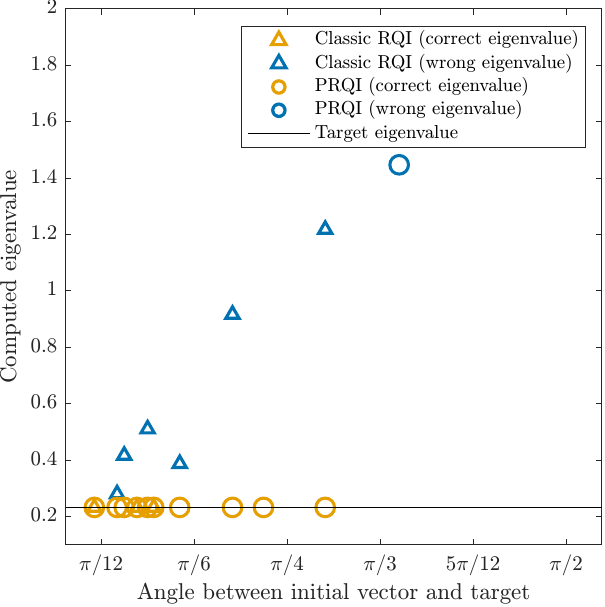}
    \caption{{\small $[1,2,1]$ matrix}}%
    \label{fig:conv_121}
  \end{subfigure}
  \hfill
  \begin{subfigure}[b]{0.45\textwidth}
    \centering
    \includegraphics[width=\textwidth]{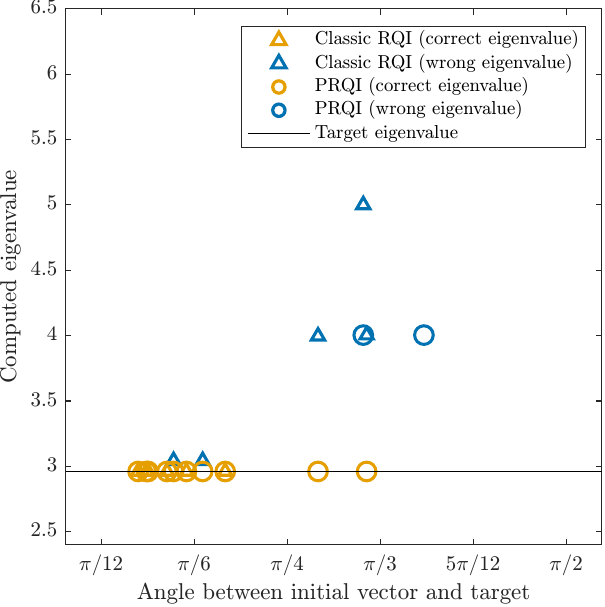}
    \caption{{\small $W_{21}^+$ Wilkinson matrix}}%
    \label{fig:conv_wilkinson}
  \end{subfigure}
  \begin{subfigure}[b]{0.45\textwidth}
    \centering
    \includegraphics[width=\textwidth]{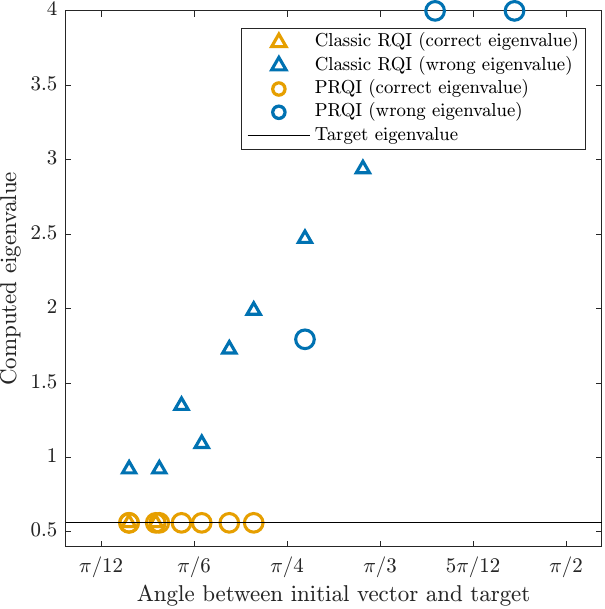}
    \caption{{\small Laplace matrix}}%
    \label{fig:conv_laplace}
  \end{subfigure}
  \hfill
  \begin{subfigure}[b]{0.45\textwidth}
    \centering
    \includegraphics[width=\textwidth]{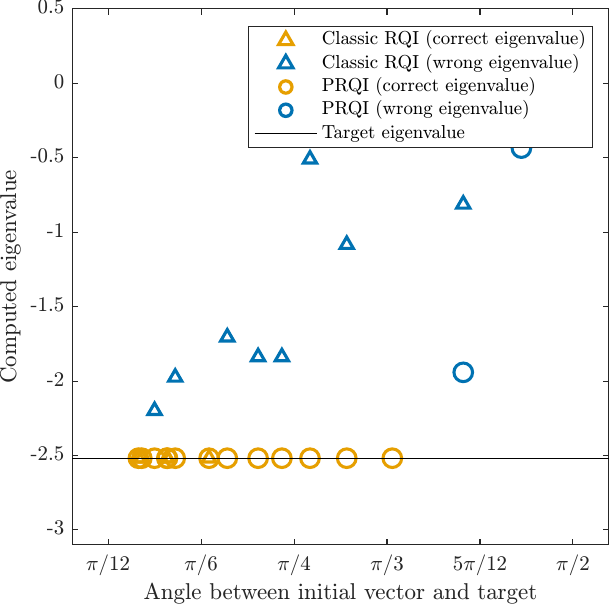}
    \caption{{\small Random matrix}}%
    \label{fig:conv_randsym}
  \end{subfigure}
  \caption{\label{fig:rqi_angle_conv_comparison} Plot of the angle between the initial vector and the target eigenvector against the computed eigenvalue for the four test matrices in Example \ref{ex:comparison2}. The target eigenvalue is depicted by the solid black line.
  While classic RQI sometimes fails even for very small angles (i.e., very accurate initial vectors), PRQI consistently produces the correct result as soon as the angle is sufficiently small.}%
\end{figure}
The plots in Figure~\ref{fig:rqi_angle_conv_comparison} show how the convergence of the classic RQI and PRQI depends on the angle between the initial vector and the target eigenvector. To generate the initial vectors, we first compute the full set of eigenvectors $\{\vec{v}_1, \dotsc, \vec{v}_n\}$ using built-in functions in MATLAB. We then draw $n$ random normally distributed real numbers $\alpha_j \sim \mathcal{N}(0,1)$, $j = 1,\dotsc,n$, and set $\vec{x}^{(0)} = \sum_{j=1}^n \alpha_j \vec{v}_j$. We randomly select one index $j^\ast$ to be the target index and throughout the experiment, we increase $\alpha_{j^\ast}$ so that the contribution of $\vec{x}^{(0)}$ in the direction of the target eigenvector $v_{j^\ast}$ becomes increasingly bigger (i.e., the angle becomes increasingly smaller).

We observe that PRQI seems to converge to the desired eigenpair once the angle between the initial vector and the target is sufficiently small. Classic RQI, on the other hand, still often converges to the wrong eigenpair for small initial error angles. In other words, our modification of classic RQI does indeed yield a method that makes more use of the initial vector instead of the initial shift.

{
For the $[1,2,1]$ test matrix we repeated the experiment with a total of $10^5$ randomly sampled initial vectors and counted the number of successful runs (i.e., the number of initial vectors that led to convergence to the target eigenpair). Here, we used the variant of PRQI that uses the squared residual norm for the magnitude of the imaginary shift $\gamma^{(k)}$. In Table~\ref{table:rqi_comparison_results} we report the portion of successful runs for different ranges of initial error angles. We also report the fraction of runs where~\eqref{eq:eigval:order} is satisfied for the initial Rayleigh quotient shift $\mu^{(0)}$, as well as the mean of the magnitude of the initial imaginary shift for PRQI.}

\begin{table}[htbp]
{
\centering
\begin{tabular}{crrrr}
\toprule
Initial error angle &  \multicolumn{2}{c}{Portion of successful runs} & \eqref{eq:eigval:order} satisfied & Mean $\gamma^{(0)}$ \\
\cmidrule(lr){2-3}
& Classic RQI & PRQI \\
\midrule
80$^{\circ}$ -- 90$^{\circ}$ & 0 \%     & 6.05 \%  & 0 \%     & 1.35 \\
70$^{\circ}$ -- 80$^{\circ}$ & 0 \%     & 31.16 \% & 0 \%     & 1.31 \\
60$^{\circ}$ -- 70$^{\circ}$ & 0.57 \%  & 92.45 \% & 0 \%     & 1.20 \\
50$^{\circ}$ -- 60$^{\circ}$ & 10.85 \% & 100 \%   & 0 \%     & 1.03 \\
40$^{\circ}$ -- 50$^{\circ}$ & 64.94 \% & 100 \%   & 0.01 \%  & 0.78 \\
30$^{\circ}$ -- 40$^{\circ}$ & 95.15 \% & 100 \%   & 21.87 \% & 0.47 \\
0$^{\circ}$ -- 30$^{\circ}$  & 100 \%   & 100 \%   & 99.30 \% & 0.26 \\
\bottomrule
\end{tabular}
\caption{{Fraction of $10^5$ random initial vectors that led to convergence to the target eigenpair of the Laplace matrix in Example 2 for different ranges of initial error angles}.}
\label{table:rqi_comparison_results}
}
\end{table}
\end{example}

\begin{example}\label{ex:sturm}
For this example we consider the following Sturm--Liouville problem
\begin{equation}
\label{eq:schroedinger}
\begin{aligned}
    -u^{\prime \prime}(x) + \left( \sin x - \frac{40}{1 + x^2} \right)u(x) &= \lambda u(x), \qquad x \in (0,\infty)\,, \\
     u(0) &= 0\,,
\end{aligned}
\end{equation}
which exhibits a band-gap structure and also has ``trapped'' discrete eigenvalues in the gaps.
These types of spectral problems arise, e.g., in the modelling of photonic crystal fibres~\cite{Kuchment2001}, and the specific problem above has been studied previously in~\cite{aceto2006, marletta2010neumann, Marletta2012}. Our goal is to compute some of the eigenvalues in the gaps, which will be complicated by the following two facts: (1) the discrete eigenvalues in the gaps accumulate at the lower end of the bands; and (2)  na\"ive numerical discretisation of~\eqref{eq:schroedinger}  
generates additional, unwanted (so-called \emph{spurious}) eigenvalues in the spectrum within close distance to and thus indistinguishable from the target eigenvalues. Nevertheless, since the eigenvectors 
have known distinct structures, PRQI is well suited to computing eigenvalues lying in one of the lower spectral gaps while avoiding convergence to a spurious eigenvalue. Since our main goal is to show that PRQI converges to one of the target eigenvalues using only coarse \emph{a priori} information about the shape of the target eigenvector, we use a simple discretisation.
To compute eigenvalues in higher spectral gaps, a discretisation with higher accuracy is required, 
since the accuracy of the eigenvalues of the discretised problem quickly deteriorates when going to higher parts of the spectrum~\cite{aceto2006}.

Let us first describe the problem and the numerical discretisation more carefully. As mentioned above, the operator in~\eqref{eq:schroedinger} has band-gap structure with very closely spaced, discrete eigenvalues in the gaps accumulating at the lower ends of the bands (see~\cite{Marletta2012} or~\cite[Chap.~4-5]{Brown2013} for details). The first two bands are given by
\begin{equation}
    \label{eq:bands}
    J_1 = [-0.37849, -0.34767],\ J_2 = [0.59480, 0.91806]\,, 
\end{equation}
and our goal is to compute some of the discrete eigenvalues in the gap between those bands. We will exploit the fact that the qualitative behaviour of the desired eigenfunctions is very different from that of the undesired ones. As depicted in Figure~\ref{fig:eigvec plots}, the target eigenfunctions decay exponentially away from the left boundary while the eigenfunctions corresponding to spurious eigenvalues decay exponentially away from the right boundary, see~\cite{Marletta2012}.
\begin{figure}[t]
\centering
\begin{subfigure}[b]{0.45\textwidth}
    \centering
    \includegraphics[width=\textwidth]{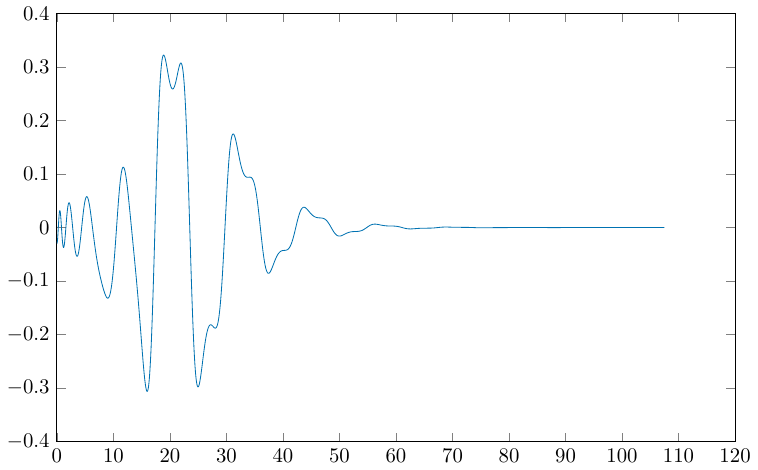}
    \caption{Target eigenfunction}
\end{subfigure}%
\begin{subfigure}[b]{0.45\textwidth}
    \centering
    \includegraphics[width=\textwidth]{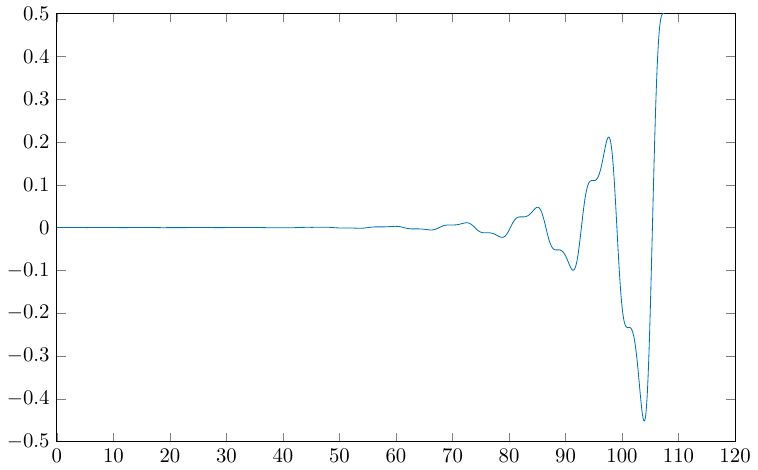}
    \caption{Eigenfunction of spurious eigenvalue}
\end{subfigure}%
\caption{Plots of one of the eigenfunctions corresponding to a target eigenvalue (left) and of the eigenfunction corresponding to an unwanted, spurious eigenvalue (right) for the Sturm-Liouville problem in \eqref{eq:schroedinger}.}\label{fig:eigvec plots}
\end{figure}

To solve~\eqref{eq:schroedinger} numerically, we discretise the equation using the finite element method, see, e.g., ~\cite{braess}.
First, the problem is truncated to the interval $[0,X]$, for $X > 0$ sufficiently large, where at the artificial boundary we impose a Neumann boundary condition $u^\prime(X) = 0$. The weak formulation of the resulting problem, obtained by multiplying with a test function $\varphi \in V \coloneqq \{ f \in H^1([0,X]) \mid f(0) = 0 \}$ and integrating by parts reads: Find $u \in V$ such that with $q(x) \coloneqq \sin x - {40}/{(1 + x^2)}$ we have
\begin{equation}
\label{eq:weakform}
    \int_0^X u^\prime \varphi^\prime + \int_0^X q u \varphi 
    = \lambda \int_0^X u \varphi\,,\qquad \text{for all } \varphi \in V.
\end{equation}

Let now $V_h \subset V$ be the space of piecewise linear finite elements with respect to the
uniform mesh $\{0 = x_1 < \dotsm < x_n = 1\}$ with (Lagrange) basis $\{\varphi_1, \dotsc, \varphi_n\}$ such that $\varphi_j(x_i) = \delta_{ij}$ for all $i,j \le n$. For the mesh width $h = x_{k+1} - x_{k}$, for all $k$, 
we used $h = 0.01$.  
Replacing $u$ in~\eqref{eq:weakform} by a function $u_h \in V_h$ and expanding in the basis such that $u_h = \sum_{j=1}^n v_j \varphi_j$
leads to an algebraic eigenvalue equation of the form
\begin{equation}
\label{eq:fe_eigval_problem}
    (A + B)\vec{v} = \lambda M \vec{v}\,,
\end{equation}
with $\vec{v} = {(v_1, \dotsc, v_n)}^T$ and the entries of the matrices $A,B,M \in \mathbb{R}^{n \times n}$ are given by
\begin{equation*}
    A_{ij} = \int_0^X \varphi_i^\prime \varphi_j^\prime\,,\qquad
    B_{ij} = \int_0^X q\varphi_i \varphi_j\,,\qquad
    M_{ij} = \int_0^X \varphi_i^\prime \varphi_j^\prime\,.
\end{equation*}

In Figure~\ref{fig:closeup_spectrum} we plot a close-up of the spectrum of the original and the perturbed problem respectively around the target eigenvalues (cf.\ Figure~\ref{fig:perturbed_spectrum}). The marked eigenvalue is the spurious one; the eigenvalues left and right of it are two of the target eigenvalues. The perturbed spectrum corresponds to the eigenproblem 
\[
  (A+B + i(M - M\vec{x}_0 {(M{\vec{x}_0})}^T)) \vec{v} = \lambda M \vec{v}\,,
\]
where $\vec{x}_0 \in \bbR^n$ is the initial vector used in PRQI, described below.
\begin{figure}
    \centering
    \includegraphics[width=0.9\textwidth]{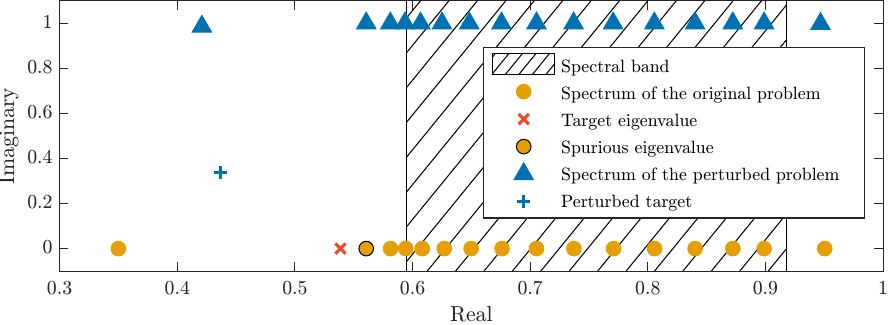}
    \caption{Part of the spectrum of the original and the perturbed Sturm--Liouville eigenvalue problem \eqref{eq:schroedinger}. The target eigenvalue in the original spectrum (red cross) is very close to the unwanted eigenvalues in the band (hatched region) and the 
    spurious eigenvalue (bordered circle). The perturbed eigenvalues (blue triangles) are raised into the complex plane and only the target eigenvalue (blue cross) stays near the real line.}%
    \label{fig:closeup_spectrum}
\end{figure}
Recall that the eigenfunctions corresponding to target eigenvalues are exponentially decaying, and, being eigenfunctions of a Sturm--Liouville operator, are also oscillating. As such, we generate the initial vectors $\vec{x}_0$ discretising and normalising (w.r.t.\ the $\|\cdot\|_M$ norm) functions $f : [0,X] \rightarrow \mathbb{R}$ that are constructed as follows:
\begin{enumerate}
    \item Choose a cutoff value $R \in (0,X)$.
    \item For $x > R$ set $f(x) = 0$ (this cutoff models the exponential decay).
    \item For $x \le R$ define $f$ to be piecewise constant and oscillating, taking values~$\pm 1$.
\end{enumerate}
Due to the Dirichlet boundary condition at $x = 0$ we further set $f(x) = 0$ on $[0,x_0]$ for some small $x_0 > 0$. Throughout our experiments we used $x_0 = 0.1$. Figure~\ref{fig:initial_vec} shows an initial vector generated by this approach for a cutoff value of $R = 35$ (which is approximately a third of the total length of the interval). Below we will vary the number of oscillations $n_{\text{osc}}$ for fixed $R$; since the oscillatory part of the function $f$ used to generated this vector consisted of three full oscillations before setting it to zero on $[0,x_0]$, we say that this vector corresponds to $n_{\text{osc}} = 3$.

Figures~\ref{fig:prqi_result} and~\ref{fig:classic_rqi_result} are the eigenvectors computed by PRQI and classic RQI, respectively using this initial vector. PRQI converged to the target eigenvalue $\lambda_{\text{PRQI}} = 0.53874$ (which is in the gap between the bands $J_1$  and $J_2$ in \eqref{eq:bands}), 
while classic RQI converged to an eigenvalue far away ($\lambda_{\text{RQI}} = 26.8648$). 
For classic RQI to converge to a value in the first spectral gap it was necessary to also set the initial shift appropriately.

\begin{figure}[t]
\centering
\begin{subfigure}[b]{0.33\textwidth}
    \centering
    \includegraphics[width=\textwidth]{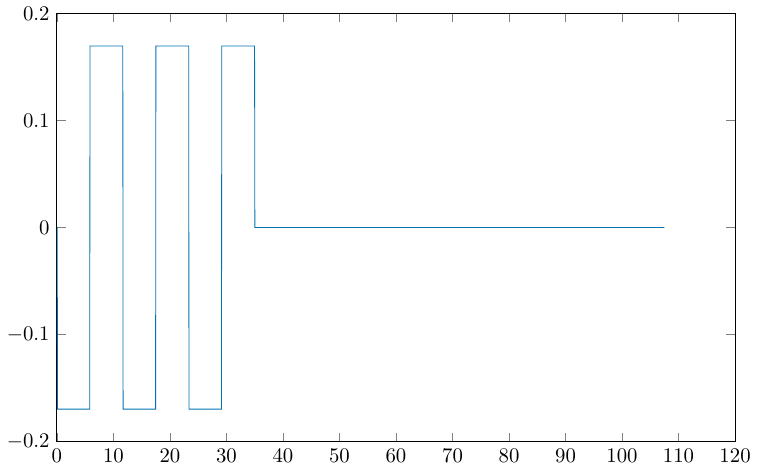}
    \caption{Initial vector}\label{fig:initial_vec}
\end{subfigure}%
\begin{subfigure}[b]{0.33\textwidth}
    \centering
    \includegraphics[width=\textwidth]{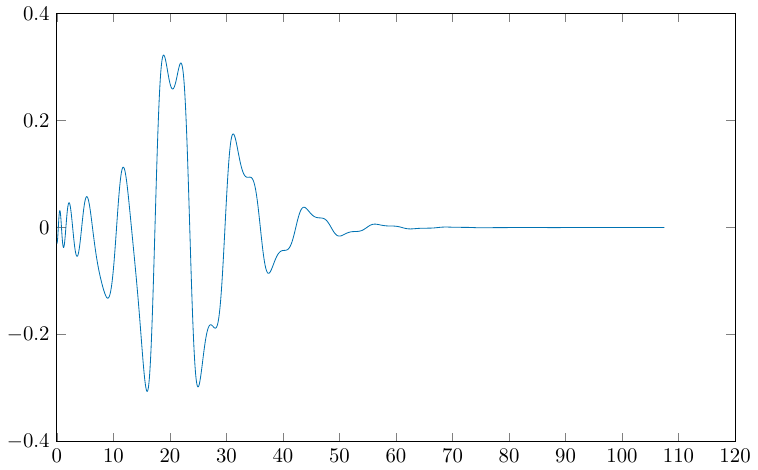}
    \caption{PRQI result}\label{fig:prqi_result}
\end{subfigure}%
\begin{subfigure}[b]{0.33\textwidth}
    \centering
    \includegraphics[width=\textwidth]{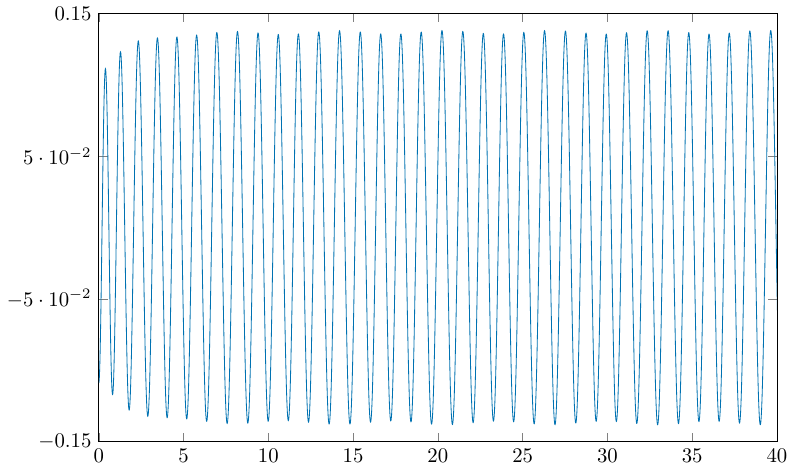}
    \caption{Classic RQI result}\label{fig:classic_rqi_result}
\end{subfigure}%
\caption{\label{fig:eigvec_comparison} The left figure shows an initial vector used to compute an eigenpair of \eqref{eq:schroedinger} in the second spectral gap with PRQI. The middle figure is the eigenvector that PRQI converged to. The right figure shows the eigenvector computed by classic RQI (we cropped the interval at $x = 40$ for better visualisation).}
\end{figure}

To target other eigenvalues in the first spectral gap, we can vary $R$ and/or $n_{\text{osc}}$. While PRQI successfully avoided convergence to a spurious eigenvalue in all our experiments, it sometimes converged to an eigenvalue in one of the spectral bands. Since these eigenvalues correspond to the essential spectrum of the original problem, the computed eigenvectors are not localised anywhere, i.e., they oscillate over the whole interval $[0,X]$. This allows us to implement a modified PRQI that early on detects convergence to an unwanted eigenvalue. Let $S \in [0,X]$ and define
\begin{equation*}
    \eta(\vec{x}^{(k)})= \frac{\|\vec{x}^{(k)} \mathds{1}_{\{x > S\}}\|}{\|\vec{x}^{(k)}\|} \in [0,1]\,,
\end{equation*}
where $\vec{x}^{(k)}$ is the current vector iterate and $\vec{x}^{(k)} \mathds{1}_{\{x > S\}} \in \bbC^n$ is a vector constructed from $\vec{x}^{(k)}$ by setting all entries at indices corresponding to function values at $x \le S$ to zero. For $S$ sufficiently large, vectors corresponding to target eigenvalues, being exponentially decaying, will satisfy $\eta(\vec{x}^{(k)}) \ll 1$. Thus, to detect convergence to an unwanted eigenvector, during each iteration of PRQI we can check if $\eta(\vec{x}^{(k)}) > \eta^\ast$ for some threshold value $\eta^\ast$ and if yes, abort the current execution (a similar criterion is proposed in~\cite{grebenkov2013,ovall2023}). We used $S = 80$ and determined $\eta^\ast$ by computing $\eta$ for a rapidly oscillating sine function. This gave $\eta \approx 0.5$, so we choose $\eta^\ast = 0.4$. 

Table~\ref{tab:results} summarises the results using this approach for a few exemplar values of $n_{\text{osc}}$ and $R$, the number of oscillations and cutoff values of the initial vectors that we used. The column ``Index''  contains the index of the eigenvalue computed by the respective method within the spectrum of the discretised equation~\eqref{eq:fe_eigval_problem} counted from the smallest eigenvalue. {In this case, we used the squared residual norm for the imaginary shift in PRQI, i.e., $\gamma^{(k)} = \|\vec{r}^{(k)}\|^2$ (which corresponds to $q=1$ in Theorem~\ref{thm:prqi:conv}).}

\begin{table}
{
  \centering
  \begin{tabular}{llcccccc}\toprule
    && \multicolumn{3}{c}{PRQI} & \multicolumn{3}{c}{Classic RQI}
    \\\cmidrule(lr){3-5}\cmidrule(lr){6-8}
    $n_{\text{osc}}$ & $R$ & Eigenvalue & Index & \# Its. & Eigenvalue & Index & \# Its. \\\midrule
    1.5 & 35 & -0.22706 &  22 &  7 & 25.06396 & 174 &  8 \\
    2 & 35 & -0.22706 &  22 & 10 & 36.44008 & 209 &  6 \\
    2.5 & 35 & -0.41034 &  10 &  8 & 43.49608 & 228 &  6 \\
    3 & 55 & -0.22706 &  22 &  9 & 34.34056 & 203 &  7 \\
  3.5 & 55 &  0.34988 &  23 &  9 & 46.25176 & 235 &  4 \\
    4 & 55 &  0.34988 &  23 &  8 & 45.06046 & 232 &  7 \\
    4.5 & 55 &  0.53874 &  24 &  8 & 59.01389 &  \hspace{-3mm}$>$250 &  5 \\
    5 & 55 &  0.58134 &  26 &  8 & 68.37970 &  \hspace{-3mm}$>$250 &  5 \\
    \bottomrule
  \end{tabular}
  \caption{\label{tab:results} Results of applying PRQI {(using $\gamma^{(k)} = \|\vec{r}^{(k)}\|^2$ in the imaginary shift)} and classic RQI to the Sturm--Liouville problem for different initial vectors parametrised by the number of oscillations $n_{\text{osc}}$ and the cutoff value $R$ . PRQI converged to one of the target eigenvalues for many different combinations of $n_{\text{osc}}$ and $R$ while classic RQI always converged to an eigenvalue far away. Furthermore, PRQI successfully skips the spurious eigenvalue $\lambda_{\text{sp}} \approx 0.56$ with Index 25, and it required on average only about two iteration more than classic RQI. 
  }%
}
\end{table}

We observe that PRQI converges to a target eigenvalue for many different initial vectors and on average requires only one iteration more than classic RQI, while
classic RQI consistently converges to an eigenvalue very far away.

The results obtained for classic RQI can be explained by computing the Rayleigh quotient for the initial vectors used. For instance, the Rayleigh quotient for the initial vector $\vec{x}^{(0)}$ corresponding to the values in the first row is given by $\raq(\vec{x}^{(0)}) \approx 24.7$. In practice, however, one would rarely use classic RQI directly for this problem and certainly not with such a bad shift. The endpoints of the spectral gaps can be computed using Floquet theory~\cite{Brown2013}, and one would therefore instead use some value inside the spectral gap in which the eigenvalue is sought as the initial shift in classic RQI. Usually a few iterations with a fixed shift (i.e., with the shifted inverse iteration) are carried out before switching to the Rayleigh quotient shift. However, as demonstrated in Example~\ref{ex:comparison3x3}, classic RQI is strongly dependent on the value of the initial shift and so this approach is not suitable for this type of problem. 

To demonstrate this, we carried out some further tests: Without any further information, one reasonable choice for the initial shift would be the midpoint of the first spectral gap $\mu^{(0)} \approx 0.1236$. However, we found that classic RQI (both when preceded by inverse iteration and when used stand-alone) usually converges to an eigenvalue close to the shift, in this case either to the eigenvalue $0.349875$ or $-0.227061$. To target the eigenvalues near the right end of the gap, we also used $\mu^{(0)} = 0.5$ as the initial shift. This, however, caused convergence to the unwanted spurious eigenvalue for some combinations of $R$ and $n_{\text{osc}}$ (e.g., $R=35$, $n_{\text{osc}} = 4$).
\end{example}

\section{Conclusion}
In this paper we have presented a modification of the classic Rayleigh Quotient Iteration with the aim to produce a method that overcomes some of the problems of RQI. The method, which we called Projected RQI (PRQI), was first introduced by perturbing the matrix using the projection onto the orthogonal complement of the initial vector and using this matrix during the iteration in RQI. We then showed that this approach can be simplified: it suffices to add an imaginary shift to the Rayleigh quotient shift in RQI. 

We then proved that PRQI is cubically convergent and we presented numerical examples that demonstrate that the convergence behaviour of PRQI is indeed more robust than that of classic RQI. In particular, PRQI seems to make more use of the information in the initial vector instead of the initial shift. We then showed how this behaviour can be used to compute certain eigenvalues in a band-gap spectrum of a Sturm--Liouville operator. It was demonstrated how a small amount of \emph{a priori} knowledge about the shape of an eigenvector can be used to construct an initial vector that precludes the convergence of PRQI to unwanted eigenpairs.

Possible avenues for future work would be to consider more examples of the Sturm--Liouville type as in Example~\ref{ex:sturm}. In~\cite{szyld} classic RQI and Inverse Iteration are combined to confine the eigenvalue iterates to a given interval. In the Sturm--Liouville example, the endpoints of the spectral gaps can be computed using Floquet theory~\cite{Brown2013}, so possibly the idea in~\cite{szyld} can be extended to include PRQI: the combination of classic RQI and Inverse Iteration would ensure that the resulting eigenvalue lies within the wanted gap while PRQI avoids convergence to a spurious eigenvalue. Combining classic RQI and PRQI could also be fruitful in other cases as a means to improve convergence speed by switching from PRQI to classic RQI (i.e., setting $\gamma = 0$) once the error residual is sufficiently small.\medskip